\documentclass{amsart}

\usepackage{setspace}
\usepackage{cite}
\usepackage{enumerate}
\usepackage{amssymb}
\usepackage{dsfont}
\usepackage[dvips]{graphicx}
\usepackage{epsfig}
\usepackage{subfigure}

\newcommand{\Cset}{\mathop{\mathds{C}}}
\newcommand{\Nset}{\mathop{\mathds{N}}}
\newcommand{\Zset}{\mathop{\mathds{Z}}}
\newcommand{\Rset}{\mathop{\mathds{R}}}

\newcommand{\st}[1]{\mathop{S(\frac{1}{#1\rho})}}
\newcommand{\sech}{\mathop{\mathrm{sech}}}
\newcommand{\PP}[2]{P^{\scriptscriptstyle{\mathcal{M}}}_{#1}(#2)}
\newcommand{\mom}{\ensuremath{{\mathcal{M}}}}
\newcommand{\Mi}{~_{\!\!\!{\mathcal{M}}}}
\newcommand{\Mscal}[2]{\langle #1, #2\rangle_{~_{\!\!\!\!{\mathcal{M}}}}}
\newcommand{\K}[1]{\ensuremath{{\mathcal{K}}^{#1}}}

\newcommand{\LL}{\ensuremath{L_{\!{\scriptscriptstyle{\mathcal{M}}}}^2}}
\newcommand{\LB}{\ensuremath{{\mathbf L}_{{\!{\scriptscriptstyle{\mathcal{M}}}}}^2}}

\newcommand{\LT}{\ensuremath{L^2_{a(\omega)} }}

\newcommand{\norm}[1]{\left\|#1\right\|_{\scriptscriptstyle{\!{\mathcal{M}}}}}
\newcommand{\noi}[1]{\left\| #1\right\|_{a(\omega)}}
\newcommand{\doti}[2]{\langle #1, #2\rangle_{a(\omega)}}
\newcommand{\Scal}[2]{\langle #1, #2
            \rangle^{\scriptscriptstyle{\mathcal{M}}}}
\newcommand{\ScalT}[2]{\langle #1, #2
            \rangle^{\scriptscriptstyle{T}}}
\newcommand{\ScalH}[2]{\langle #1, #2
            \rangle^{\scriptscriptstyle{H}}}
\newcommand{\KK}{\ensuremath{\{\K{n}\}^
{\scriptscriptstyle{\mathcal{M}}}_{n\in \Nset}}}
\newcommand{\PPP}{\ensuremath{\{\PP{n}{\omega}\}_
{n\in \Nset}}}

\newcommand{\sinc}[1]{\mathrm{sinc}\, #1}
\newcommand{\BL}[1]{\ensuremath{\mathbf{BL}(#1)}}

\newcommand{\CE}{\mathrm{CE}^{\!\scriptscriptstyle{\mathcal{M}}}}
\newcommand{\CA}{\mathrm{CA}^{\!\scriptscriptstyle{\mathcal{M}}}}
\newcommand{\mm}{\mathbi{m}}
\newcommand{\dd}{{\mathop{{D}}}}
\newcommand{\da}{{\mathop{{\rm d}a(\omega)}}}
\newcommand{\e}{\mathop{{\mathrm{e}}}}
\newcommand{\ii}{\mathop{{\mathrm{i}}}}
\newcommand{\wght}{\mathop{\mathrm{w}}}
\def\mathbi#1{\textbf{\textit #1}}
\newcommand{\CC}{\ensuremath{\mathcal{C}^{\scriptscriptstyle{\mathcal{M}}}}}
\newcommand{\CT}{\ensuremath{\mathcal{C}_2^{\scriptscriptstyle{T}}}}
\newcommand{\CH}{\ensuremath{\mathcal{C}_2^{\!\scriptscriptstyle{H}}}}
\newcommand{\Norm}[1]{\left\|#1\right
           \|^{\scriptscriptstyle{{\mathcal{M}}}}}
\newcommand{\LIM}[1]{\lim_{#1\rightarrow\infty}}
\newcommand{\PT}[2]{P_{#1}^{\scriptscriptstyle{T}}(#2)}
\newcommand{\PH}[2]{P_{#1}^{\scriptscriptstyle{H}}(#2)}
\newcommand{\Langle}{\langle\!\langle}
\newcommand{\Rangle}{\rangle\!\rangle}
\newcommand{\FT}{\mathcal{F}^{\Mi}}

\newtheorem{theorem}{theorem}[section]
\newtheorem{corollary}[theorem]{Corollary}
\newtheorem{proposition}[theorem]{Proposition}
\newtheorem{conjecture}[theorem]{Conjecture}
\newtheorem{lemma}[theorem]{Lemma}
\newtheorem{definition}[theorem]{Definition}
\newtheorem{question}{Question}

\begin{document}

\title{Chromatic Derivatives, Chromatic Expansions and
Associated Spaces}

\thanks{Some of the results from this paper were presented at \emph{UNSW
Research Workshop on Function Spaces and Applications}, Sydney,
December 2-6, 2008, and at \emph{SAMPTA 09}, Marseille, May
18-22, 2009.}

\author{Aleksandar Ignjatovi\'{c}}

\address{School of Computer Science and Engineering,
University of New South Wales {\rm and} National ICT Australia
(NICTA), Sydney, NSW 2052, Australia}

\email{ignjat@cse.unsw.edu.au}
\urladdr{www.cse.unsw.edu.au/\~{}{ignjat}}



\begin{abstract}
This paper presents the basic properties of chromatic
derivatives and chromatic expansions and provides an
appropriate motivation for introducing these notions. Chromatic
derivatives are special, numerically robust linear differential
operators which correspond to certain families of orthogonal
polynomials. Chromatic expansions are series of the
corresponding special functions, which possess the best
features of both the Taylor and the Shannon expansions. This
makes chromatic derivatives and chromatic expansions applicable
in fields involving empirically sampled data, such as digital
signal and image processing.
\end{abstract}

\maketitle

\section{Extended Abstract}

Let $\BL{\pi}$ be the space of continuous $L^2$ functions with
the Fourier transform supported within $[-\pi,\pi]$ (i.e., the
space of $\pi$ band limited signals of finite energy), and let
$P^{\scriptscriptstyle{L}}_n(\omega)$ be obtained by
normalizing and scaling the Legendre polynomials,  so that
\[
\frac{1}{2\pi}\int_{-\pi}^{\pi}
P^{\scriptscriptstyle{L}}_{\scriptstyle{n}}
(\omega)\;P^{\scriptscriptstyle{L}}_{\scriptstyle{m}}
(\omega){\rm d}\omega=\delta(m-n).
\]
We consider linear differential operators $\K{n}=(-{\ii})^{n}
P^{\scriptscriptstyle{L}}_{\scriptstyle{n}}
\left(\ii\;\frac{{\rm d}}{{\rm d} t}\right)$; for such
operators and every $f\in\BL{\pi}$,
\[
\K{n}[f](t)=\frac{1}{2\pi}\int_{-\pi}^{\pi}{\ii}^n\;
P^{\scriptscriptstyle{L}}_{\scriptstyle{n}}(\omega)
\widehat{f}(\omega){\e}^{\ii
\omega t} {\rm d}\omega.
\]

We show that for $f\in\BL{\pi}$ the values of $\K{n}[f](t)$ can
be obtained in a numerically accurate and noise robust way from
samples of $f(t)$, even for differential operators $\K{n}$ of
high order.

Operators $\K{n}$ have the following remarkable properties,
relevant for applications in digital signal processing.

\begin{proposition}
Let $f:\Rset\rightarrow\Rset$ be a restriction of any entire
function; then the following are equivalent:
\begin{enumerate}[(a)]
\item $\sum_{n=0}^{\infty}\K{n}[f](0)^2 < \infty$;
\item for all $t\in\Rset$ the sum
    $\sum_{n=0}^{\infty}\K{n}[f](t)^2$ converges, and its
    values are independent of $t\in\Rset$;
\item $f\in\BL{\pi}$.
\end{enumerate}
\end{proposition}

Moreover, the following Proposition provides local
representation of the usual norm, the scalar product and the
convolution in $\BL{\pi}$.

\begin{proposition} For all  $f, g\in\BL{\pi}$
the following sums do not depend on $t\in\Rset$, and
\begin{eqnarray}
\sum_{n=0}^{\infty}\K{n}[f](t)^2\hspace*{-2mm}&=&\hspace*{-2mm}
\int_{-\infty}^{\infty} f(x)^2 dx;\nonumber\\
\sum_{n=0}^{\infty}\K{n}[f](t)\K{n}[g](t)
\hspace*{-2mm}&=&\hspace*{-2mm} \int_{-\infty}^{\infty}
f(x)g(x) dx;\nonumber\\
\sum_{n=0}^{\infty}\K{n}[f](t) \K{n}_t[g(u-t)]
\hspace*{-2mm}&=&\hspace*{-2mm}
\int_{-\infty}^{\infty}f(x)g(u-x) dx.\nonumber
\end{eqnarray}
\end{proposition}
The following proposition provides a form of Taylor's theorem,
with the differential operators $\K{n}$ replacing the
derivatives and the spherical Bessel functions replacing the
monomials.
\begin{proposition}
Let $j_n$ be the spherical Bessel functions of the first kind;
then:
\begin{enumerate}
\item  for every entire function $f$ and for all
    $z\in\Cset$,
\[
f(z) =\sum_{n=0}^{\infty}(-1)^n\K{n}[f](0)\K{n}[j_0(\pi z)]
=\sum_{n=0}^{\infty}\K{n}[f](0)\;\sqrt{2n+1}\;j_n(\pi z);
\]
\item if  $f\in\BL{\pi}$, then the series converges
    uniformly on $\Rset$ and in $L^2$.
\end{enumerate}
\end{proposition}

We give analogues of the above theorems for very general
families of orthogonal polynomials. We also introduce some
nonseparable inner product spaces. In one of them, related to
the Hermite polynomials, functions $f_{\omega}(t)=\sin \omega
t$  for all $\omega>0$ have finite positive norms and for
\textbf{every} two distinct values $\omega_1\neq\omega_2$ the
corresponding functions $f_{\omega_1}(t)=\sin \omega_1 t$ and
$f_{\omega_2}(t)=\sin \omega_2 t$ are mutually orthogonal.
Related to the properties of such spaces, we also make the
following conjecture for families of orthonormal polynomials.
\begin{conjecture}
Let $P_n(\omega)$ be a family of symmetric positive definite
orthonormal polynomials corresponding to a moment distribution
function $a(\omega)$, \vspace*{-3mm}
\[
~\hspace*{20mm}\int_{-\infty}^{\infty}P_n(\omega)\;P_m(\omega)\;{\rm
d}a(\omega)=\delta(m-n),
\]
and let $\gamma_n>0$ be the recursion coefficients in the
corresponding three term recurrence relation for such
ortho\underline{\textbf{normal}} polynomials, i.e., such that
\[
P_{n+1}(\omega)=\frac{\omega}{\gamma_{n}}\,
P_{n}(\omega)-\frac{{\gamma_{n-1}}}{{\gamma_{n}}}\,
P_{n-1}(\omega).
\]
If $\gamma_n$ satisfy \ \
$\displaystyle{0<\lim_{n\rightarrow\infty}\frac{\gamma_n}{n^p}<\infty}$\
\  for some $0\leq p <1$, then
\[\displaystyle{0<\lim_{n\rightarrow\infty}\frac{1}{n^{1-p}}
\sum_{k=0}^{n-1}P_k(\omega)^2<\infty}
\]
for all $\omega$ in the support $sp(a)$  of  $a(\omega)$.
\end{conjecture}
Numerical tests with $\gamma_n=n^p$ for many $p\in[0,1)$
indicate that the conjecture is true.

\section{Motivation}
Signal processing mostly deals with the signals which can be represented by
continuous $L^2$ functions whose Fourier transform is supported within
$[-\pi,\pi]$; these functions form \textit{the space  $\BL{\pi}$ of $\pi$ band limited
signals of finite energy}. Foundations of classical digital signal processing
rest on the Whittaker--Kotel'nikov--Nyquist--Shannon Sampling
Theorem (for brevity the Shannon Theorem):
every signal $f\in\BL{\pi}$ can be represented using its samples
at integers and \textit{the cardinal
sine function} $\sinc t =\sin \pi t/\pi t$, as
\begin{equation}\label{nexp}
f(t)=\sum_{n=-\infty}^{\infty}f(n)\; \sinc (t-n).
\end{equation}

Such signal representation is of \textit{global nature},
because it involves samples of the signal at integers of
arbitrarily large absolute value. In fact, since for a fixed
$t$ the values of $\sinc (t-n)$ decrease slowly as $|n|$ grows,
the truncations of the above series do not provide satisfactory
local signal approximations.

On the other hand, since every signal $f\in\BL{\pi}$ is a restriction to $\Rset$
of an entire function, it can also be represented by the Taylor series,
\begin{equation}\label{taylor}
f(t)=\sum_{n=0}^{\infty}f^{(n)}(0)\;\frac{t^n}{n!}.
\end{equation}
Such a series converges uniformly on every finite interval, and
its truncations provide good local signal approximations. Since
the values of the derivatives $f^{(n)}(0)$ are determined by
the values of the signal in an arbitrarily small neighborhood
of zero, the Taylor expansion is of \textit{local nature}. In
this sense, the Shannon and the Taylor expansions are
complementary.

However, unlike the Shannon expansion, the Taylor expansion has
found very limited use in signal processing, due to several
problems associated with its application to empirically sampled
signals.
\begin{enumerate}[(I)]
\item Numerical evaluation of higher order derivatives
    of a function given by its samples is very noise
    sensitive. In general, one is cautioned against numerical
    differentiation:
    \begin{quote}``\ldots
    numerical differentiation should be avoided whenever
    possible, particularly when the data are empirical and
    subject to appreciable errors of
    observation''\cite{Hil}.\end{quote}

\item The Taylor expansion of  a signal $f\in\BL{\pi}$
    converges non-uniformly on $\Rset$; its truncations
    have rapid error accumulation when moving away from the center
    of expansion and are unbounded.

\item Since the Shannon expansion of a signal $f\in\BL{\pi}$
    converges to $f$ in $\BL{\pi}$, the action of a
    continuous linear shift invariant operator (in signal
    processing terminology, a \textit{filter})  $A$
    can be expressed using samples of
    $f$ and the \textit{impulse response\/} $A[\sinc\!]$ of
    $A$:
\begin{equation}\label{filter-ny}
A[f](t)=\sum_{n=-\infty}^{\infty}f(n)\;A[\sinc\!](t-n).
\end{equation}
In contrast, the polynomials obtained by truncating the
Taylor series do not belong to $\BL{\pi}$ and nothing
similar to \eqref{filter-ny} is true of the Taylor
expansion.
\end{enumerate}

Chromatic derivatives were introduced in \cite{IG0} to overcome
problem (I) above; the chromatic approximations were introduced
in \cite{IG00} to obtain local approximations of band-limited
signals which do not suffer from problems (II) and (III).

\subsection{Numerical differentiation of band limited signals}

To understand the problem of numerical differentiation of
band-limited signals, we consider an arbitrary $f\in \BL{\pi}$
and its Fourier transform $\widehat{f}(\omega)$; then
\[f^{(n)}(t)=\frac{1}{2\pi}\int_{-\pi}^{\pi}(\ii\omega)^n
\widehat{f}(\omega){\e}^{\ii\omega t}d\omega.\]
Figure~\ref{derivatives} (left) shows, for $n= 15$ to $n=18$,
the plots of $(\omega/\pi)^{n}$, which are, save a factor of
$\ii^n$, the symbols, or, in signal processing terminology, the
\textit{transfer functions} of the normalized derivatives
$1/\pi^n \; {\mathrm d}^n/{\mathrm d }t^n$. These plots reveal
why there can be no practical method for any reasonable
approximation of derivatives of higher orders. Multiplication
of the Fourier transform of a signal by the transfer function
of a normalized derivative of higher order obliterates the
Fourier transform of the signal, leaving only its edges, which
in practice contain mostly noise. Moreover, the graphs of the
transfer functions of the normalized derivatives of high orders
and of the same parity cluster so tightly together that they
are essentially indistinguishable; see Figure~\ref{derivatives}
(left).\footnote{If the derivatives are not normalized, their
values can be very large and are again determined essentially
by the noise present at the edge of the bandwidth of the
signal.}

However, contrary to a common belief, these facts \textit{do
not} preclude numerical evaluation of all differential
operators of higher orders, but only indicate that, from a
numerical perspective, the set of the derivatives $\{f,
f^\prime, f^{\prime\prime},\ldots\}$ is a very poor base of the
vector space of linear differential operators with real
coefficients. We now show how to obtain a base for this space
consisting of numerically robust linear differential operators.

\subsection{Chromatic derivatives}
Let polynomials $P^{\scriptscriptstyle{L}}_n(\omega)$ be
obtained by normalizing and scaling the Legendre polynomials,
so that
\begin{equation*}
\frac{1}{2\pi}\int_{-\pi}^{\pi}P^{\scriptscriptstyle{L}}_{\scriptstyle{n}}
(\omega)\;P^{\scriptscriptstyle{L}}_{\scriptstyle{m}}(\omega){\rm d}\omega=\delta(m-n).
\end{equation*}
We define operator polynomials\footnote{Thus, obtaining $\K{n}_t$
involves replacing $\omega^k$ in
$P^{\scriptscriptstyle{L}}_{\scriptstyle{n}}(\omega)$ with
$\ii^k\;{\rm d}^k/{\rm d} t^k$ for all
$k\leq n$. If $\K{n}_t$ is applied to a function of a single
variable, we drop index $t$ in $\K{n}_t$.}
\begin{equation*}\label{dop} \K{n}_t=(-{\ii})^{n}
P^{\scriptscriptstyle{L}}_{\scriptstyle{n}}\left(\ii\;\frac{{\rm d}}{{\rm d} t}\right).
\end{equation*}
Since polynomials
$P^{\scriptscriptstyle{L}}_{\scriptstyle{n}}(\omega)$ contain
only powers of the same parity as $n$, operators $\K{n}$ have
real coefficients, and it is easy to verify that
\begin{equation*}
\K{n}_t[{\e}^{\ii\omega t}]=
{\ii}^nP^{\scriptscriptstyle{L}}_{\scriptstyle{n}}(\omega)\,{\e}^{\ii\omega t}.
\end{equation*}
Consequently, for $f\in\BL{\pi}$,
\[
\K{n}[f](t)=\frac{1}{2\pi}\int_{-\pi}^{\pi}{\ii}^n
P^{\scriptscriptstyle{L}}_{\scriptstyle{n}}(\omega)\widehat{f}(\omega)\;
{\e}^{\ii\omega t}{\rm d}\omega.
\]
In particular, one can show that
\begin{equation}\label{cdersinc}
\K{n}[\sinc](t)=(-1)^n\;\sqrt{2n+1}\;{\mathrm j}_{n}(\pi t),
\end{equation}
where ${\mathrm j}_n(x)$ is the spherical Bessel function of
the first kind of order $n$. Figure~\ref{derivatives} (right)
shows the plots of
$P^{\scriptscriptstyle{L}}_{\scriptstyle{n}}(\omega)$, for $n=
15$ to $n=18$, which are the transfer functions (again save a
factor of $\ii^n$) of the corresponding operators $\K{n}$.
Unlike the transfer functions of the (normalized) derivatives
$1/\pi^n\;{\rm d}^n/{\rm d} t^n$, the transfer functions of the
chromatic derivatives $\K{n}$ form a family of well separated,
interleaved and increasingly refined comb filters. Instead of
obliterating, such operators encode the features of the Fourier
transform of the signal (in signal processing terminology, the
\textit{spectral features} of the signal). For this reason, we
call operators $\K{n}$ the {\em chromatic derivatives\/}
associated with the Legendre polynomials.

\begin{figure}
    \begin{center}
     \includegraphics[width=4.7in]{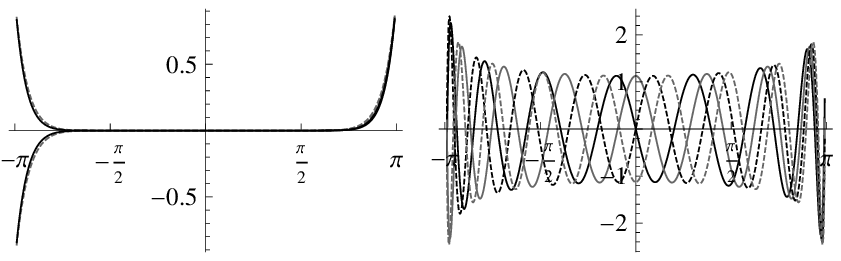}
      \caption{Graphs of $\left(\frac{\omega}{\pi}\right)^n$ (left) and of
      $P^{\scriptscriptstyle{L}}_{\scriptstyle{n}}(\omega)$
      (right) for $n=15-18$.}
      \label{derivatives}
    \end{center}
\end{figure}
Chromatic derivatives can be accurately and robustly evaluated
from samples of the signal taken at a small multiple of the
usual Nyquist rate. Figure~\ref{remezPlainLegendre15} (left)
shows the plots of the transfer function of a
\textit{transversal filter} given by
$\mathcal{T}_{15}[f](t)=\sum_{k=-64}^{64}c_{k}\,f(t+k/2)$
(gray), used to approximate the chromatic derivative
$\K{15}[f](t)$, and the transfer function of $\K{15}$ (black).
The coefficients $c_k$ of the filter were obtained using the
Remez exchange method \cite{Opp}, and satisfy $|c_k|< 0.2,
(-64\leq k\leq 64$). The filter has 129 taps, spaced two taps
per Nyquist rate interval, i.e., at a distance of $1/2$. Thus,
the transfer function of the corresponding ideal filter
$\K{15}$ is
$P^{\scriptscriptstyle{L}}_{\scriptstyle{15}}(2\omega)$ for
$|\omega|\leq \pi/2$, and zero outside this interval. The
pass-band of the actual transversal filter is 90\% of the
bandwidth $[-\pi/2,\pi/2]$. Outside the transition regions
$[-11\pi/20,-9\pi/20]$ and $[9\pi/20,11\pi/20]$ the error of
approximation is less than $1.3\times 10^{-4}$.

Implementations of filters for
operators $\K{n}$ of orders $0\leq n\leq 24$ have been tested
in practice and proved to be both accurate and noise robust, as
expected from the above considerations.
\begin{figure}
    \begin{center}
      \includegraphics[width=5in]{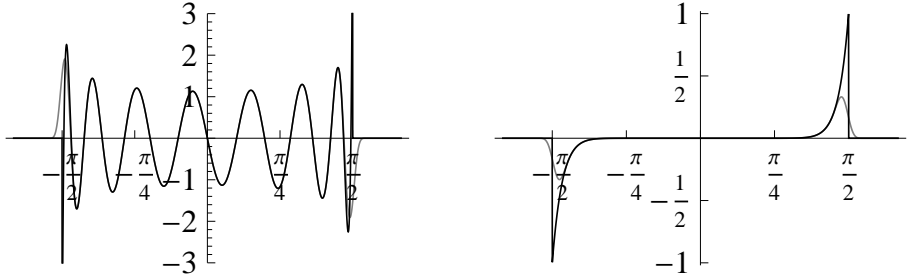}
      \caption{Transfer functions of
$\K{15}$ (left, black) and $d^{15}/dt^{15}$ (right, black) and
of their transversal filter approximations (gray).}
      \label{remezPlainLegendre15}
    \end{center}
\end{figure}

For comparison, Figure~\ref{remezPlainLegendre15} (right) shows
the transfer function of a transversal filter obtained by the
same procedure and with the same bandwidth constraints, which
approximates the (normalized) ``standard" derivative
$(2/\pi)^{15}\; {\rm d}^{15}/{\rm d} t^{15}$  (gray) and the
transfer function of the ideal filter (black). The figure
clearly indicates that such a transversal filter is of no
practical use.

Note that \eqref{nexp} and \eqref{cdersinc} imply that
\begin{equation}\label{no}
\K{k}[f](t)=\sum_{n=-\infty}^{\infty}f(n)\;\K{k}[\sinc\!](t-n)
=\sum_{n=-\infty}^{\infty}f(n)\;(-1)^k\;\sqrt{2k+1}\;{\mathrm j}_k(\pi(t-n)).
\end{equation}
However, in practice, the values of $\K{k}[f](t)$, especially
for larger values of $k$, \textit{cannot} be obtained from the
Nyquist rate samples using truncations of \eqref{no}. This is
due to the fact that functions $\K{k}[\sinc\!](t-n)$ decay very
slowly as $|n|$ grows; see Figure~\ref{i15} (left). Thus, to
achieve any accuracy, such a truncation would need to contain
an extremely large number of terms.
\begin{figure}
    \begin{center}
      \includegraphics[width=5in]{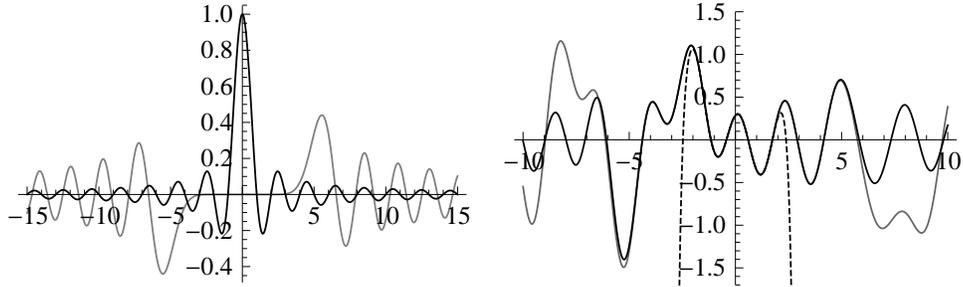}
      \caption{\textsc{left}: Oscillatory behavior of $\sinc(t)$
      (black), and $\K{15}[\sinc](t)$ (gray);
      \textsc{right}: A signal $f\in\BL{\pi}$ (gray) and its
      chromatic and Taylor approximations (black, dashed)}
      \label{i15}
    \end{center}
\end{figure}
On the other hand, this also means that signal information
present in the values of the chromatic derivatives of a signal
obtained by sampling an appropriate filterbank  at an instant
$t$ is \textit{not redundant} with information present in the
Nyquist rate samples of the signal in any reasonably sized
window around $t$, which is a fact suggesting that chromatic
derivatives could enhance standard signal processing methods
operating on Nyquist rate samples.

\subsection{Chromatic expansions}
The above shows that numerical evaluation of the chromatic
derivatives associated with the Legendre polynomials does not
suffer problems which precludes numerical evaluation of the
``standard'' derivatives of higher orders. On the other hand,
the chromatic expansions, defined in Proposition~\ref{app}
below, were conceived as a solution to problems associated with
the use of the Taylor expansion.\footnote{Propositions stated
in this Introduction are special cases of general propositions
proved in subsequent sections.}

\begin{proposition}\label{app}
Let $\K{n}$ be the chromatic derivatives associated with the
Legendre polynomials, let ${\mathrm j}_n$ be the spherical
Bessel function of the first kind of order $n$, and let $f$  be
an arbitrary entire function; then for all $z,u\in\Cset$,
\begin{eqnarray}
f(z)&=&\sum_{n=0}^{\infty}\;K^n[f](u)\;  K^n_u[\sinc (z-u)]\label{CEL}\\
&=&\sum_{n=0}^{\infty}(-1)^n\;K^n[f](u)\;  K^n[\sinc ](z-u)\label{CEL1}\\
&=&\sum_{n=0}^{\infty}K^n[f](u)\;\sqrt{2n+1}\;  {\mathrm j}_n (\pi(z-u))
\end{eqnarray}
If  $f\in\BL{\pi}$, then the series converges uniformly on
$\Rset$ and in the space $\BL{\pi}$.
\end{proposition}

The series in \eqref{CEL} is called \textit{the chromatic
expansion of $f$ associated with the Legendre polynomials\/}; a
truncation of this series is called a \textit{chromatic
approximation\/} of $f$. As the Taylor approximation, a
chromatic approximation is also a local approximation; its
coefficients are the values of differential operators
$\K{m}[f](u)$ at a single instant $u$, and for all $k\leq n$,
\begin{equation*}
f^{(k)}(u)=\frac{{\rm d}^k}{{\rm d}t^k}
\left[\sum_{m=0}^{n}\K{m}[f](u) \;\K{m}_u[\sinc(t-u)]\right]_{t=u}.
\end{equation*}

Figure~\ref{i15} (right) compares the behavior of the chromatic
approximation (black) of a signal $f\in\BL{\pi}$ (gray) with
the behavior of its Taylor approximation (dashed).
Both approximations are of order 16. The signal $f(t)$
is defined using the Shannon expansion, with
samples $\{f(n)\ : \ |f(n)|<1,\  -32\leq n\leq 32\}$ which were
randomly generated. The plot
reveals that, when approximating a signal $f\in\BL{\pi}$, a
chromatic approximation has a much gentler error accumulation
when moving away from the point of expansion than the Taylor
approximation of the same order.

Unlike the monomials which appear in the Taylor formula,
functions $\K{n}[\sinc\!](t)=(-1)^n\;\sqrt{2n+1}\;{\mathrm
j}_n(\pi t)$ belong to $\BL{\pi}$ and satisfy
$|\K{n}[\sinc\!](t)|\leq 1$ for all $t\in\Rset$. Consequently,
the chromatic approximations also belong to $\BL{\pi}$ and are
bounded on $\Rset$.

Since by Proposition~\ref{app} the chromatic approximation of a
signal $f\in\BL{\pi}$ converges to $f$ in $\BL{\pi}$, if $A$ is
a filter, then $A$ commutes with the differential operators
$\K{n}$ and thus for every $f\in \BL{\pi}$,
\begin{equation}\label{filter-ce}
A[f](t)=\sum_{n=0}^{\infty}(-1)^n\;\K{n}[f](u)\;
\K{n}[A[\,\sinc ]](t-u).
\end{equation}
A comparison of \eqref{filter-ce} with \eqref{filter-ny}
provides further evidence that, while local just like the
Taylor expansion, the chromatic expansion associated with the
Legendre polynomials possesses the features that make the
Shannon expansion so useful in signal processing. This,
together with numerical robustness of chromatic derivatives,
makes chromatic approximations applicable in fields involving
empirically sampled data, such as digital signal and image
processing.

\subsection{A local definition of the scalar product in \BL{\pi}}
Proposition~\ref{fun} below demonstrates another
remarkable property of the chromatic derivatives associated with
the Legendre polynomials.

\begin{proposition}\label{fun}
Let $f:\Rset\rightarrow\Rset$ be a restriction of an arbitrary
entire function; then the following are equivalent:
\begin{enumerate}[(a)]
\item $\sum_{n=0}^{\infty}\K{n}[f](0)^2 < 0$;
\item for all $t\in\Rset$ the sum $\sum_{n=0}^{\infty}\K{n}[f](t)^2$
converges, and its values are independent of $t\in\Rset$;
\item $f\in\BL{\pi}$.
\end{enumerate}
\end{proposition}

The next proposition is relevant for signal
processing because it provides \textit{local representations}
of the usual norm, the scalar product and the convolution in $\BL{\pi}$,
respectively, which are defined globally, as improper integrals.

\begin{proposition}\label{local-space} Let $\K{n}$ be the chromatic
derivatives associated with the (rescaled and normalized)
Legendre polynomials, and $f, g\in\BL{\pi}$. Then the following
sums do not depend on $t\in\Rset$ and satisfy
\begin{eqnarray}
\sum_{n=0}^{\infty}K^n[f](t)^2\hspace*{-2mm}&=&\hspace*{-2mm}
\int_{-\infty}^{\infty} f(x)^2 dx;\\
\sum_{n=0}^{\infty}K^n[f](t)K^n[g](t)
\hspace*{-2mm}&=&\hspace*{-2mm} \int_{-\infty}^{\infty}
f(x)g(x) dx;\\
\sum_{n=0}^{\infty}K^n[f](t) K^n_t[g(u-t)]
\hspace*{-2mm}&=&\hspace*{-2mm}
\int_{-\infty}^{\infty}f(x)g(u-x) dx.
\end{eqnarray}
\end{proposition}

\subsection{}
We finish this introduction by pointing to a close relationship
between the Shannon expansion and the chromatic expansion
associated with the Legendre polynomials. Firstly, by
\eqref{CEL1},
\begin{equation}\label{fn}
f(n)=\sum_{k=0}^{\infty} \;\K{k}[f](0)\;(-1)^k
\K{k}[\sinc\!](n).
\end{equation}
Since $\K{n}[\sinc\!](t)$ is an even function for even $n$ and
odd for odd $n$,  \eqref{no} implies
\begin{equation}\label{kn}
\K{k}[f](0)=\sum_{n=-\infty}^{\infty}\;f(n)\; (-1)^k
K^k[\sinc\!](n).
\end{equation}
Equations \eqref{fn} and \eqref{kn} show that the coefficients
of the Shannon expansion of a signal
-- the samples $f(n)$, and the coefficients of the chromatic
expansion of the signal -- the simultaneous samples of the
chromatic derivatives $\K{n}[f](0)$, are related by an
orthonormal operator defined by the infinite matrix
\[
\left[(-1)^k \K{k}\left[ \sinc\right](n)\ :\ k\in\Nset,
n\in\Zset\right]=\left[\sqrt{2k+1}\;{\mathrm j}_{k}(\pi n)\ :\ k\in\Nset,
n\in\Zset\right].
\]

Secondly, let $\mathcal{S}_u[f(u)]=f(u+1)$ be the unit shift
operator in the variable $u$ ($f$ might have other parameters).
The Shannon expansion for the set of sampling points $\{u+n\ :\
{n\in \Zset}\}$ can be written in a form analogous to
the chromatic expansion, using operator polynomials
$\mathcal{S}_u^n=\mathcal{S}_u\circ\ldots\circ\mathcal{S}_u$,
as
\begin{eqnarray}
f(t)&=&\sum_{n=0}^{\infty}f(u+n)\;\sinc (t-(u+n))\\
&=&\sum_{n=0}^{\infty} \mathcal{S}^{n}_u[f](u)\;
\mathcal{S}^{n}_u[\sinc (t-u)];\label{fancyN}
\end{eqnarray}
compare now \eqref{fancyN} with \eqref{CEL}. Note that the
family of operator polynomials $\{\mathcal{S}^{n}_u\}_{n\in
\Zset}$ is also an orthonormal system, in the sense that their
corresponding transfer functions $\{\e^{\ii n\;
\omega}\}_{n\in\Zset}$ form an orthonormal system in
$L^2[-\pi,\pi]$. Moreover, the transfer functions of the
families of operators $\{\K{n}\}_{n\in \Nset}$ and
$\{\mathcal{S}^{n}\}_{n\in \Zset}$, where $\K{n}$ are the
chromatic derivatives associated with the Legendre polynomials,
are orthogonal on $[-\pi,\pi]$ with respect to the
same, constant weight $\mathrm{w}(\omega)= 1/(2\pi)$.

In this paper we consider chromatic derivatives and chromatic
expansions which correspond to some very general families of
orthogonal polynomials, and prove generalizations of the above
propositions, extending our previous work \cite{IG5}.\footnote{
Chromatic expansions corresponding to general families of
orthogonal polynomials were first considered in \cite{CH}.
However,  Proposition 1 there is false; its attempted proof
relies on an incorrect use of the Paley-Wiener Theorem. In
fact, function $F_g$ defined there need not be extendable to an
entire function, as it can be shown using Example 4 in
Section~\ref{examples} of the present paper.} However, having
in mind the form of expansions \eqref{CEL} and \eqref{fancyN},
one can ask a more general (and somewhat vague) question.
\begin{question}\label{Q2}  What are the operators $A$ for which
there exists a family of operator polynomials $\{P_n(A)\}$,
orthogonal under a suitably defined notion of  orthogonality,
such that for an associated function
$\mathbf{m}_{\scriptscriptstyle{A}}(t)$,
\begin{equation*}
f(t)=\sum_{n}P_n(A)[f](u)\;P_n(A)
[\mathbf{m}_{\scriptscriptstyle{A}}(t-u)]
\end{equation*}
for all functions from a corresponding (and significant) class
$\mathcal{C}_{\scriptscriptstyle{A}}$?
\end{question}

\section{Basic Notions}

\subsection{Families of orthogonal polynomials} Let
$\mom:\mathcal{P}_{\omega}\rightarrow \Rset$ be a linear
functional on the vector space $\mathcal{P}_{\omega}$ of real
polynomials in the variable $\omega$ and let
${\mu}_{n} = \mom(\omega^{n})$. Such $\mom$ is a \textit{moment
functional}, ${\mu}_{n}$ is the \textit{moment} of
$\mom$ of order $n$ and the \textit{Hankel
determinant} of order $n$ is given by
\begin{equation*}\nonumber
\Delta_n=\left|\begin{array}{ccc} \mu_0&\ldots &\mu_n\\
\mu_1&\ldots &\mu_{n+1}\\
\ldots&\ldots&\ldots\\
\mu_{n}&\ldots&\mu_{2n}
\end{array}\right|.
\end{equation*}

The moment functionals $\mom$ which we consider are assumed to be:
\begin{enumerate}[(i)]
\item\label{c1} positive definite, i.e., $\Delta_{n}>0$ for all $n$;
such functionals also satisfy
$\mu_{2n}> 0$;
\item\label{c2} symmetric, i.e., $\mu_{2n+1}=0$ for all
    $n$;
\item\label{c3} normalized, so that $\mom(1)=\mu_0=1$.
\end{enumerate}

For functionals  $\mom$  which satisfy the above three conditions
there exists a family \PPP\  of polynomials with real
coefficients, such that
\begin{enumerate}[(a)]
\item \PPP\ is an orthonormal system with respect to \mom,
    i.e., for all $m,n$,
    \[\mom(\PP{m}{\omega}\,\PP{n}{\omega})=\delta(m-n);\]
\item each polynomial $\PP{n}{\omega}$ contains only
powers of $\omega$ of the same parity as $n$;
\item $\PP{0}{\omega}=1$.
\end{enumerate}

A family of polynomials is the family of
orthonormal polynomials
corresponding to a symmetric positive definite moment
functional \mom\ just in case there exists a sequence of reals
$\gamma_n>0$ such that for all $n>0$,
\begin{equation}\label{poly}
\PP{n+1}{\omega}=\frac{\omega}{\gamma_{n}}\,
\PP{n}{\omega}-\frac{\gamma_{n-1}}{\gamma_{n}}\,
\PP{n-1}{\omega}.
\end{equation}
If we set $\gamma_{-1}=1$ and $\PP{-1}{\omega}=0$,
then \eqref{poly} holds for $n=0$ as well.

We will make use of the Christoffel-Darboux equality for
orthogonal polynomials,
\begin{equation}\label{CDP}
(\omega-\sigma)\sum_{k=0}^{n}
\PP{{k}}{\omega}\PP{{k}}{\sigma}
= \gamma_n
(\PP{{n+1}}{\omega}\PP{{n}}{\sigma}-\PP{{n+1}}{\sigma}
\PP{{n}}{\omega}),
\end{equation}
and of its consequences obtained by setting $\sigma=-\omega$
in \eqref{CDP} to get
\begin{equation}\label{equal}
\omega\left(\sum_{k=0}^{n}\PP{{2k+1}}{\omega}^2-
\sum_{k=0}^{n}\PP{{2k}}{\omega}^2\right)=
\gamma_{2n+1}\,\PP{{2n+2}}{\omega}\,
\PP{{2n+1}}{\omega},
\end{equation}
and by by letting $\sigma\rightarrow\omega$ in \eqref{CDP} to
get
\begin{equation}\label{sumsquares}
\sum_{k=0}^{n}\PP{{k}}{\omega}^2 =  \gamma_n \,
(\PP{{n+1}}{\omega}^\prime
\PP{{n}}{\omega}-\PP{{n+1}}{\omega}
\PP{{n}}{\omega}^\prime).
\end{equation}

For  every positive definite moment functional \mom\ there
exists a non-decreasing bounded function  $ a (\omega)$, called
a \textit{moment distribution function}, such that for the associated
Stieltjes integral we have
\begin{equation}\label{l3}
\int_{-\infty}^{\infty}
\omega^{n}\,\da=\mu_n
\end{equation}
and such that for the corresponding family of polynomials $\PPP$
\begin{equation}\label{polyortho}
\int_{-\infty}^{\infty} \PP{n}{\omega}\,\PP{m}{\omega}
\,\da=\delta(m-n).
\end{equation}

We denote by $\LT$  the Hilbert space of functions
$\varphi:\Rset\rightarrow \Cset$  for which the
Lebesgue-Stieltjes integral
$\int_{-\infty}^{\infty}|\varphi(\omega)|^2\, \da$ is finite,
with the scalar product defined by
$\doti{\alpha}{\beta}=\int_{-\infty}^{\infty}\alpha(\omega)
\,\overline{\beta(\omega)}\, \da$, and with the
corresponding norm denoted by $\noi{\varphi}$.

We define a function $\mm:\Rset\rightarrow\Rset$ as
\begin{equation}\label{mm}
\mm(t)=\int_{-\infty}^{\infty}{\e}^{{\ii}\omega t }\da.
\end{equation}
Since
\begin{eqnarray*}
\int_{-\infty}^{\infty}|(\ii\omega)^{n}
\,{\e}^{{\ii}\,\omega t} | \da
\leq\left(\int_{-\infty}^{\infty}\!\!\omega^{2n} \da
\int_{-\infty}^{\infty}\!\!\da\right)^{1/2}
\!\!\!=\sqrt{\mu_{2n}}<\infty,
\end{eqnarray*}
we can differentiate \eqref{mm} under the integral sign any
number of times, and obtain that for all s$n$,
\begin{eqnarray}\label{dm}
&\mm^{(2n)}(0)=(-1)^n\mu_{2n};& \\
&\mm^{(2n+1)}(0)=0.&
\end{eqnarray}

\subsection{The chromatic derivatives}

Given a moment functional $\mom$ satisfying
conditions \eqref{c1} -- \eqref{c3} above, we associate with
\mom\ a family of linear
differential operators $\KK$ defined by the operator
polynomial\footnote{Thus, to obtain $\K{n}$, one replaces
$\omega^k$ in $\PP{n}{\omega}$ by $\ii^k D^k_t$, where
$\dd^k_t[f]=\frac{\mathrm{d}^k}{\mathrm{d}t^k}f(t)$. We use
the square brackets to indicate the arguments of operators
acting on various function spaces. If $A$ is a linear differential
operator, and if a function $f(t,\vec{w})$ has parameters
$\vec{w}$, we write $A_{t}[f]$ to distinguish the variable $t$
of differentiation; if $f(t)$ contains only variable $t$, we
write $A[f(t)]$ for $A_{t}[f(t)]$ and $\dd^k[f(t)]$ for
$\dd^k_t [f(t)]$. }
\begin{eqnarray*}
\K{n}_t=\frac{1}{\ii^n}\; P_{n}^{\scriptscriptstyle{\mathcal{M}}}
\left(\ii \dd_t\right),
\end{eqnarray*}
and call them \textit{the chromatic derivatives associated with}
\mom.  Since \mom\ is symmetric, such operators have real
coefficients and satisfy
the recurrence
\begin{equation}\label{three-term}
\K{n+1}=\frac{1}{{\gamma_{n}}}\,(\dd\circ
\K{n})+\frac{{\gamma_{n-1}}}{{\gamma_{n}}}\, \K{n-1},
\end{equation}
with the same coefficients $\gamma_n>0$ as in \eqref{poly}.
Thus,
\begin{equation}\label{iwt}
\K{n}_t[e^{\ii\omega t}]=
{\ii}^n\PP{n}{\omega}\,e^{\ii\omega t}
\end{equation}
and
\begin{equation}
\K{n}[\mm](t)=\int_{-\infty}^{\infty}
{\ii}^{n}\,\PP{n}{\omega}\,
{\e}^{{\ii}\omega t} \da.
\end{equation}
The basic properties of
orthogonal polynomials imply that for all $m,n,$
\begin{equation}\label{orthonorm}
(\K{n}\circ \K{m})[\mm](0)
=(-1)^{n}\delta(m-n),
\end{equation}
and, if $m < n$ or if $m-n$ is odd, then
\begin{equation}\label{bkn}
(\dd^m\circ\K{n})[\mm](0)=0.
\end{equation}

The following Lemma corresponds to the Christoffel-Darboux
equality for orthogonal polynomials and has a similar proof
which  uses \eqref{three-term} to represent the left hand side
of \eqref{C-D} as a telescoping sum.

\begin{lemma}[\!\!\cite{IG5}]\label{CD} Let $\KK$ be the family
of chromatic derivatives associated with a moment
functional $\mom$, and let $f,g\in C^\infty$; then
\begin{equation}\label{C-D}
 \dd\left[\sum_{m=0}^{n} \K{m}[f]\,\K{m}[g]\right]= {\gamma_{n}}\,
(\K{n+1}[f]\,\K{n}[g]+\K{n}[f]\,\K{n+1}[g]).
\end{equation}
\end{lemma}

\subsection{Chromatic expansions} Let $f$ be infinitely
differentiable at a real or complex $u$; the formal series
\begin{eqnarray}\label{cex}
\CE[f,u](t)&=&\sum_{k=0}^{\infty}\K{k}[f](u)\;\K{k}_u[\mm(t-u)]\\
&=&\sum_{k=0}^{\infty}(-1)^k\K{k}[f](u)\;\K{k}[\mm](t-u)\nonumber
\end{eqnarray}
is called the \textit{chromatic expansion\/} of $f$  associated
with \mom, centered at $u$, and
\begin{equation*}
\CA[f,n,u](t)=\sum_{k=0}^{n}(-1)^k\K{k}[f](u)\K{k}[\mm](t-u)
\end{equation*}
is the \textit{chromatic approximation\/} of $f$ of order $n$.

From \eqref{orthonorm} it follows that the chromatic
approximation $\CA[f,n,u](t)$ of order $n$ of $f(t)$ for all
$m\leq n$ satisfies
\begin{eqnarray*}
\K{m}_{t} [\CA[f,n,u](t)]\big |_{t=u}&=&\sum_{k=0}^{n} \,(-1)^k\K{k}
[f](u)\, (\K{m}\circ\K{k})[\mm](0)\ =\ \K{m}[f](u).
\end{eqnarray*}

Since $\K{m}$ is a linear combination of derivatives $\dd^k$
for $k\leq m$, also $f^{(m)}(u) = \dd^m_t[\CA[f,n,u](t)]\big |_{t=u}$
for all $m\leq n$.
In this sense, just like the Taylor approximation, a chromatic
approximation is a local approximation. Thus, for all $m\leq n$,
\begin{equation}\label{k-to-d}
f^{(m)}(u) = \dd^m_t[\CA[f,n,u](t)]\big |_{t=u} \!\!=
\sum_{k=0}^{n}(-1)^k \, \K{k}[f](u)\,(\dd^{m}\circ \K{k})[\mm](0).
\end{equation}
Similarly, since
${\dd^m_t}\!\left[\sum_{k=0}^{n}f^{(k)}(u){(t-u)^k}/{k!}\right]\big
|_{t=u}\!\!=f^{(m)}(u)$ for $m\leq n$, we also have
\begin{equation}\label{d-to-k}
\K{m}[f](u)= \K{m}_{t}
\left[\sum_{k=0}^{n}f^{(k)}(u){(t-u)^k}/{k!}\right]_{t=u}\!\!=\!
\sum_{k=0}^{n}f^{(k)}(u)\,
\K{m}\left[t^k/k!\right](0).
\end{equation}

Equations \eqref{k-to-d} and \eqref{d-to-k} for $m=n$  relate
the standard and the chromatic bases of the vector space space of linear
differential operators,
\begin{eqnarray}
\dd^{n} &=& \sum_{k=0}^{n}(-1)^k \, (\dd^{n} \circ \K{k})[\mm](0)\;
\K{k};
\label{inverse}\\
\K{n}&=&\sum_{k=0}^{n} \K{n}\left[t^k/k!\right](0)\;
\dd^k.\label{direct}
\end{eqnarray}

Note that, since for  $j>k$ all powers of $t$ in
$\K{k}\left[{t^j}/{j!}\right]$ are positive, we have
\begin{equation}\label{zero-mon} j>k\ \ \Rightarrow \ \
\K{k}\left[{t^j}/{j!}\right](0)=0.
\end{equation}

\section{Chromatic Moment Functionals}\label{sec3}

\subsection{} We now introduce the broadest class of moment
functionals which we study.

\begin{definition}
Chromatic moment functionals are symmetric positive definite
moment functionals for which the sequence
$\{\mu_n^{1/n}/n\}_{n\in\Nset}$ is bounded.
\end{definition}
If \mom\ is chromatic, we set
\begin{equation}\label{limsup}
\rho=\limsup_{n\rightarrow\infty}
\left(\frac{\mu_n}{n!}\right)^{1/n}=\e\;
\limsup_{n\rightarrow\infty} \frac{\mu_n^{1/n}}{n}<\infty.
\end{equation}

\begin{lemma}\label{meuw}
Let \mom\ be a chromatic moment functional and $\rho$ such that
\eqref{limsup} holds. Then for every $\alpha$ such that
$0\leq \alpha < 1/\rho$ the corresponding moment distribution
$a(\omega)$ satisfies
\begin{equation}\label{weight}
\int_{-\infty}^{\infty}{\e}^{\alpha|\omega|}\da <\infty.
\end{equation}
\end{lemma}
\begin{proof}
For all $b>0$,
$\int_{-b}^{b}{\e}^{\alpha|\omega|}\da =
\sum_{n=0}^{\infty}{\alpha^n}/{n!}\int_{-b}^{b}|\omega|^n\da.$
For even $n$ we have $\int_{-b}^{b}\omega^n\da\leq\mu_n.$
For odd $n$ we have $|\omega|^n<1+\omega^{n+1}$ for
all $\omega$, and thus
$\int_{-b}^{b}|\omega|^n\da <\int_{-b}^{b}\da +
\int_{-b}^{b}\omega^{n+1}\da\leq 1 + \mu_{n+1}.$
Let $\beta_n=\mu_n$ if $n$ is even, and $\beta_n=1+\mu_{n+1}$
if $n$ is odd. Then also
$\int_{-\infty}^{\infty}{\e}^{\alpha|\omega|}\da \leq
\sum_{n=0}^{\infty}{\alpha^n}\beta_n/{n!}$.
Since $\limsup_{n\rightarrow\infty}(\beta_n/n!)^{1/n}=\rho$
and $0\leq \alpha<1/\rho$, the last sum converges to a
finite limit.
\end{proof}

On the other hand, the proof of Theorem 5.2
in \S II.5 of \cite{Fr} shows that if
\eqref{weight} holds for some $\alpha> 0$, then
\eqref{limsup} also holds for some $\rho\leq 1/\alpha$.
Thus, we get the following Corollary.
\begin{corollary}
A symmetric positive definite moment functional is chromatic
just in case for some $\alpha > 0$ the corresponding
moment distribution function $a(\omega)$ satisfies \eqref{weight}.
\end{corollary}

\noindent \textbf{Note:} \textit{For the remaining part of this
section we assume that \mom\ is a chromatic moment functional.}\\

For every $a>0$, we let $S(a)=\{z\in
\Cset\;:\;|\mathrm{Im}(z)|<a\}$. The following Corollary
directly follows
from Lemma \ref{meuw}.
\begin{corollary}\label{euw}
If $u\in \st{2}$, then ${\e}^{\ii u\; \omega}\in\LT$.
\end{corollary}

We now extend function $\mm(t)$ given by \eqref{mm} from
$\Rset$ to the complex strip $\st{}$.
\begin{proposition}\label{anaz}
Let for $z\in \st{}$,
\begin{equation}\label{mmz}
\mm(z)=\int_{-\infty}^{\infty}{\e}^{{\ii}\omega z }\da.
\end{equation}
Then $\mm(z)$ is analytic on the strip $\st{}$.
\end{proposition}
\begin{proof}
Fix $n$ and let $z=x+\ii y$ with $|y|<1/\rho$; then for
every $b>0$,
\begin{equation*}
\int_{-b}^{b}|
(\ii \omega)^n{\e}^{{\ii}\omega z }|\da
\leq\int_{-b}^{b}|\omega|^n {\e}^{|\omega y|}\da
=\sum_{k=0}^{\infty}\frac{|y|^k}{k!}\int_{-b}^{b} |\omega|^{n+k}\da.
\end{equation*}
As in the proof of Lemma \ref{meuw}, we let
$\beta_k=\mu_{n+k}$ for even values of $n+k$, and
$\beta_k=1+\mu_{n+k+1}$ for odd values of $n+k$; then the above
inequality implies
\[\int_{-\infty}^{\infty}|
(\ii \omega)^n{\e}^{{\ii}\omega z }|\da
\leq\sum_{k=0}^{\infty}\frac{|y|^k\beta_k}{k!},\] and it is
easy to see that for every fixed $n$,
$\limsup_{k\rightarrow\infty} (\beta_k/k!)^{1/k}=\rho$.
\end{proof}

\begin{proposition}\label{f-phi}
Let $\varphi(\omega)\in \LT$; we can define a
corresponding function
$f_{\varphi}:\st{2}\rightarrow \Cset$ by
\begin{equation}\label{f-int}
f_{\varphi}(z)=\int_{-\infty}^{\infty}\varphi(\omega)
e^{{\ii}\omega z }\,\da.
\end{equation}
Such $f_{\varphi}(z)$ is analytic on $\st{2}$ and for all $n$
and $z\in \st{2}$,
\begin{equation}\label{fourier-int-der}
\K{n}[f_{\varphi}](z)=\int_{-\infty}^{\infty}
{\ii}^{n}\,\PP{n}{\omega}\,
\varphi(\omega)\, {\e}^{{\ii}\omega z} \da.
\end{equation}
\end{proposition}

\begin{proof}
Let $z = x +\ii y$, with $|y|<1/(2\rho)$. For every $n$ and
$b>0$ we have
\begin{eqnarray*}
&&\hspace*{-20mm}\int_{-b}^{b}| (\ii \omega)^n
\varphi(\omega){\e}^{{\ii}\omega z }|\da\\
&\leq&\int_{-b}^{b}|\omega|^n |\varphi(\omega)
|{\e}^{|\omega y|}\da\\
&\leq&\sum_{k=0}^{\infty}\frac{|y|^k}{k!}\int_{-b}^{b}
|\omega|^{n+k}|\varphi(\omega)| \da\\
&\leq&\sum_{k=0}^{\infty}\frac{|y|^k}{k!}
\left(\int_{-b}^{b} \omega^{2n+2k}\da\int_{-b}^{b}
|\varphi(\omega)|^2 \da\right)^{1/2}.
\end{eqnarray*}
Thus, also \[\int_{-\infty}^{\infty}| (\ii \omega)^n
\varphi(\omega){\e}^{{\ii}\omega z }|\da\leq
\noi{\varphi(\omega)}
\sum_{k=0}^{\infty}\frac{|y|^k}{k!}\sqrt{\mu_{2n+2k}}.\]
The claim now follows from the fact that
$\limsup_{k\rightarrow\infty}
{\sqrt[2k]{\mu_{2n+2k}}}/{\sqrt[k]{k!}}=2\rho$ for every fixed
$n$.
\end{proof}

\begin{lemma}\label{complete}
If \mom\ is chromatic, then $\PPP$ is a complete system in
$\LT$.
\end{lemma}

\begin{proof}
Follows from a theorem of  Riesz (see, for example, Theorem 5.1
in \S II.5 of \cite{Fr}) which asserts that if
$\liminf_{n\rightarrow\infty}
\left({\mu_n}/{n!}\right)^{1/n}<\infty$, then $\PPP$ is a
complete system in $\LT$.\footnote{Note that we need the
stronger condition $\limsup_{n\rightarrow\infty}
\left({\mu_n}/{n!}\right)^{1/n}<\infty$ to insure that function
$\mm(z)$ defined by \eqref{mmz} is analytic on a strip
(Proposition~\ref{anaz}).}

\end{proof}

\begin{proposition}\label{some}
Let $\varphi(\omega)\in\LT$; if for some fixed $u\in\st{2}$ the
function $\varphi(\omega){\e}^{{\ii}\omega u}$ also belongs to
$\LT$, then in $\LT$ we have
\begin{equation}\label{ft-expand}
\varphi(\omega)e^{{\ii}\omega u }=\sum_{n=0}^{\infty}(-{\ii})^{n}
\K{n}[f_{\varphi}](u)\,\PP{n}{\omega},
\end{equation}
and for $f_{\varphi}(z)$ given by \eqref{f-int} we have
\begin{equation}\label{inde}
\sum_{n=0}^{\infty} |\K{n}[f_{\varphi}](u)|^2=
\noi{\varphi(\omega){\e}^{{\ii}\omega u}}^{2} <\infty.
\end{equation}
\end{proposition}
\begin{proof}
By Proposition \ref{f-phi}, if $u\in\st{2}$,
then equation \eqref{fourier-int-der} holds
for the corresponding $f_{\varphi}$ given by \eqref{f-int}. If
also $\varphi(\omega){\e}^{{\ii}\omega u}\in\LT$, then
\eqref{fourier-int-der} implies that
\begin{equation}\label{proj}
\doti{\varphi(\omega){\e}^{{\ii}\omega u}}{\PP{n}{\omega}}=
(-{\ii})^n\K{n}[f_{\varphi}](u).
\end{equation}
Since $\{\PP{n}{\omega}\}_{n\in\Nset}$ is a complete
orthonormal system in $\LT$, \eqref{proj} implies
\eqref{ft-expand}, and Parseval's Theorem implies \eqref{inde}.
\end{proof}

\begin{corollary}\label{constant}
For every $\varphi(\omega)\in\LT$ and every $u\in\Rset$,
equality \eqref{ft-expand} holds and
\begin{equation}\sum_{n=0}^{\infty} |\K{n}[f_{\varphi}](u)|^2
=\noi{\varphi(\omega)}^{2}.
\end{equation}
Thus, the sum $\sum_{n=0}^{\infty} |\K{n}[f_{\varphi}](u)|^2$
is independent of $u\in \Rset$.
\end{corollary}
\begin{proof}
If $u\in\Rset$, then $\varphi(\omega){\e}^{{\ii}\omega u
}\in\LT$ and $\noi{\varphi(\omega){\e}^{{\ii}\omega u
}}^{2}=\noi{\varphi(\omega)}^{2}$.
\end{proof}

\begin{corollary}\label{sum-squares}
Let $\varepsilon>0$; then for all $u\in S(\frac{1}{2\rho}-\varepsilon)$
\begin{equation}
\sum_{n=0}^{\infty}
|\K{n}[\mm](u)|^2
< \noi{{\e}^{(\frac{1}{2\rho}-\varepsilon) |\omega| }}^2<\infty.
\end{equation}
\end{corollary}
\begin{proof}
Corollary \ref{euw} implies that we can apply Proposition
\ref{some} with $\varphi(\omega)=1$, in which case
$f_{\varphi}(z)=\mm(z)$, and, using Lemma \ref{meuw}, obtain
\begin{equation*}
\sum_{n=0}^{\infty} |\K{n}[\mm](u)|^2=
\noi{{\e}^{{\ii}\omega u}}^{2}\leq
\noi{{\e}^{|\mathrm{Im}(u)\,\omega|}}^{2}<
\noi{{\e}^{(\frac{1}{2\rho}-\varepsilon) |\omega| }}^2<\infty.
\end{equation*}
\end{proof}
\begin{definition}
$\LB$ is the vector space of functions $f:\st{2}
\rightarrow \Cset$ which are analytic on $\st{2}$ and
satisfy $\sum_{n=0}^{\infty}|\K{n}[f](0)|^2<\infty$.
\end{definition}

\begin{proposition}\label{bijection}
The mapping
\begin{equation}\label{ft0}
f(z)\mapsto \varphi_f(\omega)=\sum_{n=0}^{\infty}
(-{\ii})^{n}\K{n}[f](0)\,\PP{n}{\omega}
\end{equation}
is an isomorphism between the vector spaces $\LB$ and $\LT$,
and its inverse is given by \eqref{f-int}.
\end{proposition}
\begin{proof} Let $f\in\LB$; since $\sum_{n=0}^{\infty} |\K{n}[f](0)|^2<\infty$,
the function $\varphi_f(\omega)$ defined by \eqref{ft0} belongs
to $\LT$. By Proposition~\ref{f-phi}, $f_{\varphi_f}$ defined
from $\varphi_f$ by \eqref{f-int} is analytic on $\st{2}$ and
by Proposition~\ref{some} it satisfies
$\varphi_f(\omega)=\sum_{n=0}^{\infty}(-{\ii})^{n}
\K{n}[f_{\varphi_f}](0)\,\PP{n}{\omega}.$ By the uniqueness of
the Fourier expansion of $\varphi_f(\omega)$ with respect to
the system \PPP\ we have $\K{n}[f](0)=\K{n}[f_{\varphi_f}](0)$
for all $n$. Thus, $f(z)=f_{\varphi_f}(z)$ for all $z\in
\st{2}$.
\end{proof}
Proposition \ref{bijection} and Corollary \ref{constant}
imply the following Corollary.
\begin{corollary}\label{independent}
For all $f\in\LB$ and all $t\in\Rset$ the sum
$\sum_{n=0}^{\infty}|\K{n}[f](t)|^2$ converges and is
independent of $t$.
\end{corollary}

\begin{definition}
For every $f(z)\in\LB$ we call the corresponding
$\varphi_f(\omega)\in \LT$ given by equation \eqref{ft0} the
\mom-Fourier-Stieltjes transform of $f(z)$ and denote it by
$\FT[f](\omega)$.
\end{definition}

Assume that $ a (\omega)$ is absolutely continuous; then
$a^\prime(\omega)= \wght(\omega)$ almost everywhere for some
non-negative weight function $\wght(\omega)$. Then
\eqref{f-int} implies
\begin{equation*}
f(z)=\int_{-\infty}^{\infty}\FT[f](\omega)
\;{\e}^{\ii\omega
z}\wght(\omega)\;{\rm d}\omega.
\end{equation*}
This implies the following Proposition.
\begin{proposition}[\!\!\cite{IG5}]\label{weight-space}
Assume that a function $f(z)$ is analytic on the strip $\st{2}$
and that it has a Fourier transform
$\widehat{f}(\omega)=\int_{-\infty}^{\infty}f(t)\,
{\e}^{-\ii\omega t} {\rm d} t$ such that
$f(z)=\frac{1}{2\pi}\int_{-\infty}^{\infty}
\widehat{f}(\omega)\;{\e}^{\ii\omega z}\;{\rm d}\omega$ for all
$z\in\st{2}$; then $f(z)\in\LB$ if and only if
$\int_{-\infty}^{\infty}|\widehat{f}(\omega)|^2\;
\wght(\omega)^{-1}\;{\rm d}\omega<\infty,$ in which case
$\widehat{f}(\omega)=2\pi\; \FT[f](\omega)\wght(\omega)$.
\end{proposition}

\subsection{Uniform convergence of chromatic expansions}

The Shannon expansion of an $f\in\BL{\pi}$ is obtained by
representing its Fourier transform $\widehat{f}(\omega)$ as
series of the trigonometric polynomials; similarly, the
chromatic expansion of an $f\in\LB$ is obtained by representing
$\FT[f](\omega)$ as a series of orthogonal polynomials \PPP.
\begin{proposition}\label{unif-con}
Assume $f\in\LB$; then for all $u\in\Rset$ and $\varepsilon>0$,
the chromatic series $\CE[f,u](z)$ of $f(z)$ converges to
$f(z)$ uniformly on the  strip
$S(\frac{1}{2\rho}-\varepsilon)$.
\end{proposition}
\begin{proof}
Assume $u\in\Rset$; by applying \eqref{fourier-int-der} to
$\mm(z-u)$ we get that for all $z\in\st{2}$,
\begin{equation}\label{exp-CA}\CA[f,n,u](z)=
\int_{-\infty}^{\infty}\sum_{k=0}^{n}(-{\ii})^{k}
\K{k}[f](u)\,
\PP{k}{\omega}{\e}^{\ii \omega (z-u)}\da.
\end{equation}

Since $f\in\LB$, Proposition \ref{bijection} implies
$\FT[f](\omega)\in\LT$. Corollary~\ref{constant} and equation
\eqref{ft-expand} imply that in $\LT$
\[
\FT[f](\omega)\;{\e}^{\ii\omega u}=
\sum_{k=0}^{\infty}(-{\ii})^{k}
\K{k}[f](u)\,\PP{k}{\omega}.
\]
Thus,
\begin{eqnarray}\label{f-ex}
f(z)=\int_{-\infty}^{\infty}\sum_{k=0}^{\infty}(-{\ii})^{k}
\K{k}[f](u)\,\PP{k}{\omega}{\e}^{\ii \omega (z-u)}\da.
\end{eqnarray}
Consequently, from \eqref{exp-CA} and \eqref{f-ex},
\begin{eqnarray*}
 |f(z)-\CA[f,n,u](z)|\\
&&\hspace*{-50mm}\leq\int_{-\infty}^{\infty}\left|
\sum_{k=n+1}^{\infty}(-{\ii})^{k}
\K{k}[f](u)\,\PP{k}{\omega}{\e}^{\ii \omega (z-u)}
\right|\da\nonumber\\
&&\hspace*{-50mm}\leq\left(\!\int_{-\infty}^{\infty}\!
\left|\sum_{k=n+1}^{\infty}(-{\ii})^{n}
\K{n}[f](u)\,\PP{n}{\omega}\right|^2\!\da
\int_{-\infty}^{\infty}|{\e}^{\ii\omega(z-u)}|^2\da\!\right)^{1/2}
\end{eqnarray*}

For $z\in S(\frac{1}{2\rho}-\varepsilon)$ we have
\begin{equation}\label{convergence}
|f(z)-\CA[f,n,u](z)|\leq\!\left(\sum_{k=n+1}^{\infty}|
\K{n}[f](u)|^2
\!\!\int_{-\infty}^{\infty}
{\e}^{(\frac{1}{\rho}-2\varepsilon)|\omega|}\da\right)^{1/2}.
\end{equation}
Consequently, Lemma \ref{meuw} and Corollary \ref{independent}
imply that $\CE[f,u](z)$ converges to $f(z)$ uniformly on
$S(\frac{1}{2\rho}-\varepsilon)$.
\end{proof}

\begin{proposition}\label{rep} Space $\LB$ consists precisely
of functions of the form
$f(z)=\sum_{n=0}^{\infty}a_{n} \K{n}[\mm](z)$
where $a=\Langle
a_n\Rangle_{n\in\Nset}$ is a complex sequence in $l^2$.
\end{proposition}
\begin{proof}
Assume $a\in l^2$; by Proposition \ref{sum-squares}, for every
$\varepsilon>0$, if $z\in S(\frac{1}{2\rho}-\varepsilon)$ then
\begin{eqnarray*}
\sum_{n=k}^{\infty}|a_{n}\K{n}[\mm](z)|&\leq&
\left(\sum_{n=k}^{\infty}|a_{n}|^2\sum_{n=k}^{\infty}
|\K{n}[\mm](z)|^2\right)^{1/2}\\
&\leq &\left(\sum_{n=k}^{\infty}|a_{n}|^2\right)^{1/2}
\noi{{\e}^{(\frac{1}{2\rho}-\varepsilon) |\omega| }},
\end{eqnarray*}
which implies that the series converges absolutely and uniformly on
$S(\frac{1}{2\rho}-\varepsilon)$. Consequently,
$f(z)=\sum_{n=0}^{\infty}a_{n} \K{n}[\mm](z)$ is analytic on
$\st{2}$, and
\begin{eqnarray*}
\K{m}[f](0)=
\sum_{n=0}^{\infty}a_n(\K{m}\circ\K{n})[\mm](0)=(-1)^ma_m.
\end{eqnarray*}
Thus, $\sum_{n=0}^{\infty} |\K{n}[f](0)|^2=\sum_{n=0}^{\infty}
|a_n|^2<\infty$ and so $f(z)\in \LB$.  Proposition
\ref{unif-con} provides the opposite direction.
\end{proof}
Note that Proposition \ref{rep} implies that for every
$\varepsilon>0$ functions $f(z)\in\LB$ are bounded on the strip
$S(\frac{1}{2\rho}-\varepsilon)$ because
\begin{equation*}
|f(z)|\leq \left(\sum_{n=0}^{\infty}|\K{n}[f](0)|^2\right)^{1/2}
\noi{{\e}^{(\frac{1}{2\rho}-\varepsilon) |\omega|}}.
\end{equation*}

\subsection{A function space with a locally defined scalar product}

\begin{definition}
$\LL$ is the space of functions $f(t): \Rset \mapsto \Cset$
obtained from functions in
$\LB$ by restricting their domain to $\Rset$.
\end{definition}

Assume that $f,g\in\LL$; then \eqref{proj} implies that for all
$u\in\Rset$,
\begin{eqnarray*}
\sum_{n=0}^{\infty} \K{n}[f](u)\overline{\K{n}[g](u)}&=&
\doti{\FT[f](\omega){\e}^{{\ii}\omega u}}{\FT[g](\omega)
{\e}^{{\ii}\omega u}}\\
&=&\doti{\FT[f](\omega)}{\FT[g](\omega)}.
\end{eqnarray*}

Note that for all $t,u\in \Rset$,
\begin{eqnarray*}
&&\sum_{n=0}^{n}\K{k}[f](u)\,\K{k}_u[g(t-u)]\\
&&\hspace*{10mm}=\int_{-\infty}^{\infty}\sum_{k=0}^{n}
\K{k}[f](u)\,({-\ii})^{k}\, \PP{k}{\omega}\ \FT[g](\omega)\,
{\e}^{{\ii}\omega (t-u)}\da.
\end{eqnarray*}
By \eqref{ft-expand}, the sum $\sum_{k=0}^{n}({-\ii})^{k}
\K{k}[f](u)\, \PP{k}{\omega}$ converges in $\LT$ to
$\FT[f](\omega)\;{\e}^{{\ii}\omega u}$. Since $\FT[g](\omega)\,
{\e}^{{\ii}\omega (t-u)}\in\LT$, and since
\[\left|\int_{-\infty}^{\infty}\FT[g](\omega)\,
\FT[f](\omega)\,{\e}^{{\ii}\omega
t}\da\right|<\noi{\FT[f](\omega)}\noi{\FT[g](\omega)},\]
we have that for all  $t,u\in \Rset$,
\begin{equation*}
\sum_{k=0}^{\infty}\K{k}[f](u)\, \K{k}_u[g(t-u)]
=\int_{-\infty}^{\infty}\FT[g](\omega)\,
\FT[f](\omega)\,{\e}^{{\ii}\omega t}\da<\infty.
\end{equation*}

\begin{proposition}[\!\!\cite{IG5}]\label{local-space-gen}
We can introduce locally defined scalar product, an associated
norm and a convolution of functions in \LL\ by the following
sums which are independent of $u\in\Rset:$
\begin{eqnarray}
\norm{f}^2&=&\sum_{n=0}^{\infty}|K^n[f](u)|^2=
\noi{\FT[f](\omega)}^2;\label{nor}\\
\Mscal{f}{g}&=&\sum_{n=0}^{\infty}K^n[f](u)\overline{K^n[g](u)}\\
&=&\doti{\FT[f](\omega)}{\FT[g](\omega)};\label{scl}\nonumber\\
(f \ast_{\Mi} g)(t)&=&\sum_{n=0}^{\infty}K^n[f](u)
K^n_u[g(t-u)]\label{convolution}\\
&=&\int_{-\infty}^{\infty}
{\FT[f](\omega)}\;{\FT[g](\omega)}{\e}^{{\ii}\omega t}
\da.\nonumber
\end{eqnarray}
\end{proposition}

Letting $g(t)\equiv\mm(t)$ in \eqref{convolution}, we get
$(f\ast_{\Mi}\mm)(t)=\CE[f,u](t)=f(t)$ for all
$f(t)\in \LL$, while by setting $u=0$, $u=t$ and $u=t/2$ in
\eqref{convolution}, we get the following lemma.
\begin{lemma}[\!\cite{IG5}]\label{con}
For every $f,g\in \LL$ and for every $t\in\Rset,$
\begin{eqnarray}\label{symm}
\sum_{k=0}^{\infty}
(-1)^k\,\K{k}[f](t)\,\K{k}[g](0)
&=&\sum_{k=0}^{\infty}(-1)^k\,\K{k}[f](0)\, \K{k}[g](t)\nonumber\\
&=&\sum_{k=0}^{\infty}(-1)^k\,\K{k}[f](t/2)\, \K{k}[g](t/2).\nonumber
\end{eqnarray}
\end{lemma}

Since $\mm(z)$ is analytic on $\st{}$, so are $\K{n}[\mm](z)$
for all $n$; thus, since by \eqref{orthonorm}
$\sum_{m=0}^{\infty}(\K{n}\circ\K{m})[\mm](0)^2=1$, we have
$\K{n}[\mm](t)\in \LL$ for all $n$. Let $u$ be a fixed real
parameter; consider functions $B^{n}_{u}(t)=\K{n}_u[\mm(t-u)]=
(-1)^n\K{n}[\mm](t-u)\in\LL$. Since
\begin{eqnarray}\label{orthob}
&&\Mscal{B^{n}_{u}(t)}{B^{m}_{u}(t)}\\
&&\hspace*{5mm}=\sum_{k=0}^{\infty}(\K{k}_t\circ\K{n}_u)[\mm(t-u)]
(\K{k}_t\circ\K{m}_u)[\mm(t-u)]\nonumber\\
&&\hspace*{5mm}=\sum_{k=0}^{\infty}(-1)^{m+n}(\K{k}\circ\K{n})[\mm](t-u)
(\K{k}\circ\K{m})[\mm](t-u)\Big|_{t=u}\nonumber\\
&&\hspace*{5mm}=\delta(m-n),\nonumber
\end{eqnarray}
the family $\{B^{n}_{u}(t)\}_{n\in\Nset}$ is orthonormal in
$\LL$ and for all $n\in\Nset$ and all $t\in\Rset$,
\begin{equation}\label{sumsq}
\sum_{k=0}^{\infty}(\K{k}\circ\K{n})[\mm](t)^2=1.
\end{equation}

By \eqref{orthonorm}, for $f\in\LL$,
\begin{eqnarray}\label{projf}
\Mscal{f}{\K{n}_u[\mm(t-u)]}&=&
\sum_{k=0}^{\infty}
\K{k}[f](t)(\K{k}_t\circ\K{n}_u)[\mm(t-u)]\Big|_{t=u}\nonumber\\
&=& \sum_{k=0}^{\infty}(-1)^n
\K{k}[f](u)(\K{k}\circ\K{n})[\mm](0)\nonumber\\
&=&\K{n}[f](u).
\end{eqnarray}
\begin{proposition}[\!\!\cite{IG5}]\label{in-LL}
The chromatic expansion $\CE[f,u](t)$ of $f(t)\in\LL$ is the
Fourier series of $f(t)$ with respect to the orthonormal system
$\{\K{n}_u[\mm(t-u)]\}_{n\in\Nset}$.  The chromatic
expansion converges to $f(t)$
in \LL; thus, $\{\K{n}_u[\mm(t-u)]\}_{n\in\Nset}$ is a complete
orthonormal base of \LL.
\end{proposition}
\begin{proof}
Since  $\K{k}_t[f(t)-\CE[f,n,u](t)]|_{t=u}$ equals $0$ for
$k\leq n$ and equals $\K{k}[f](u)$ for $k>n$,
$\norm{f-\CE[f,n,u]}=
\sum_{k=n+1}^{\infty}\K{k}[f](u)^2\rightarrow 0$.
\end{proof}
Note that using \eqref{sumsq} with $n=0$ we get
\begin{eqnarray}
|f(t)-\CE[f,n,u](t)|&\leq&\sum_{k=n+1}^{\infty}
|\K{k}[f](u)\K{k}[\mm](t-u)|\nonumber\\
&&\hspace*{-25mm}\leq\left(\sum_{k=n+1}^{\infty}\K{k}[f](u)^2
\sum_{k=n+1}^{\infty}\K{k}[\mm](t-u)^2\right)^{1/2}\nonumber\\
&&\hspace*{-25mm}=\left(\sum_{k=n+1}^{\infty}\K{k}[f](u)^2\right)^{1/2}
\left(1-\sum_{k=0}^{n}\K{k}[\mm](t-u)^2\right)^{1/2}.\label{error}
\end{eqnarray}

Let
\begin{eqnarray*}
E_n(t)=\left(1-\sum_{k=0}^{n}\K{k}[\mm](t)^2\right)^{1/2};
\end{eqnarray*}
then, using Lemma \ref{CD}, we have
\begin{eqnarray*}
E_n^\prime(t)=\gamma_n\;\K{n+1}[\mm](t)\;\K{n}[\mm](t)
\left(1-\sum_{k=0}^{n}\K{k}[\mm](t)^2\right)^{-1/2}.
\end{eqnarray*}
Since $(D^{k}\circ\K{n})[\mm](0)=0$ for all $0\leq k\leq n-1$,
we get that $E_n^{(k)}(t)=0$ for all $k\leq 2n+1$. Thus,
$E_{n}(0)=0$ and $E_{n}(t)$ is very flat around $t=0$, as the
following graph of $E_{15}(t)$ shows, for the particular case
of the chromatic derivatives associated with the Legendre
polynomials. This explains why chromatic expansions provide
excellent local approximations of signals $f\in\BL{\pi}$.
\begin{figure}[h]
\begin{center}
\includegraphics[width = 2.5in]{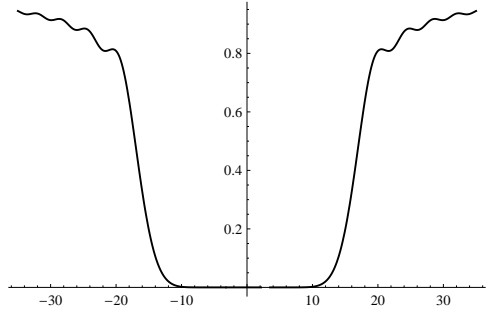}
\caption{Error bound $E_{15}(t)$ for the chromatic
approximation of order $15$ associated with the Legendre polynomials.}
\end{center}
\end{figure}

\subsection{Chromatic expansions and linear operators}
Let $A$ be a linear operator on $\LL$ which is continuous with
respect to the norm $\norm{f}$. If $A$ is shift invariant,
i.e., if for every fixed $h$, $A[f(t+h)]=A[f](t+h)$ for all
$f\in \LL$, then $A$ commutes with differentiation on \LL\ and
\begin{equation*}
A[f](t)=\sum_{n=0}^{\infty} \,(-1)^{n}\K{n} [f](u)\, \K{n}[A[\mm]](t-u)
=(f\ast_{\Mi}A[\mm])(t).
\end{equation*}
Consequently, the action of such $A$ on any function in $\LL$
is uniquely determined by $A[\mm]$, which plays the role of
the \emph{impulse response} $A[\sinc]$ of a \emph{continuous time
invariant linear system} in the standard signal processing
paradigm based on Shannon's expansion.

Note that if $A$ is a continuous linear operator $A$ on \LL\ such that
$(A\circ D^n)[\mm](t)=(D^n\circ A)[\mm](t)$ for all $n$, then
Lemma~\ref{con} implies that for every $f\in \LL$,
\begin{eqnarray*}
A[f](t)&=& \sum_{n=0}^{\infty} \,(-1)^{n}\K{n} [f](0)\,
\K{n}[A[\mm]](t) \\
&=&\sum_{n=0}^{\infty} \,(-1)^{n} \K{n}[A[\mm]](0)\, \K{n} [f](t).
\end{eqnarray*}
Since operators $\K{n}[f](t)$ are shift invariant, such $A$ must be also
shift invariant.

\subsection{A geometric interpretation}
For every particular value of $t\in \Rset$ the mapping of \LL\
into $l^2$ given by $f\mapsto f_t=\Langle
\K{n}[f](t)\Rangle_{n\in\Nset}$ is unitary isomorphism which
maps the base of \LL,  consisting of vectors
$B^{k}(t)=(-1)^k\K{k}[\mm(t)]$, into vectors $B^{k}_{t}=
\Langle(-1)^k(\K{n}\circ\K{k})[\mm(t)]\Rangle_{n\in\Nset}$.
Since the first sum in \eqref{orthob} is independent of $t$, we
have $\langle B^{k}_{t},B^{m}_{t}\rangle=\delta(m-k)$, and
\eqref{projf} implies $\langle
f_t,B^{k}_{t}\rangle=\K{k}[f](0)$. Thus, since
$\sum_{k=0}^{\infty}\K{k}[f](0)^2< \infty$, we have
$\sum_{k=0}^{\infty}\K{k}[f](0)\,B^{k}_{t}\in l^2$ and
\begin{eqnarray*}
\sum_{k=0}^{\infty}\langle f_t,B^{k}_{t}\rangle\; B^{k}_{t}&=&
\sum_{k=0}^{\infty}\K{k}[f](0)\,B^{k}_{t}\\&=&
\Big\langle\!\!\Big\langle\sum_{k=0}^\infty\K{k}[f](0)
(-1)^k(\K{n}\circ\K{k})[\mm(t)]\Big\rangle\!\!
\Big\rangle_{n\in\Nset}.
\end{eqnarray*}
Since for $f\in\LL$ the chromatic series of $f$ converges
uniformly on $\Rset$, we have $K^{n}[f](t)=
\sum_{k=0}^{\infty}K^{k}[f](0)
(-1)^k(\K{n}\circ\K{k})[\mm(t)].$ Thus,
\begin{equation*}
\sum_{k=0}^{\infty}\langle f_t,B^{k}_{t}\rangle\; B^{k}_{t}=
\sum_{k=0}^{\infty}K^{k}[f](0)\; B^{k}_{t}=
\Langle \K{n}[f](t)\Rangle_{n\in \Nset}=f_t.
\end{equation*}

Thus, while the coordinates of $f_t=\Langle \K{n}[f](t)\Rangle_{n\in
\Nset}$ in the usual base of $l^2$ vary with $t$, the
coordinates of $f_t$ in the bases $\{B^{k}_{t}\}_{k\in \Nset}$
remain the same as $t$ varies.

We now show that $\{B^{n}_{t}\}_{n\in\Nset}$ is the moving
frame of a helix $H:\Rset\mapsto l_2$.

\begin{lemma}[\!\!\cite{IG5}]\label{continuous}
Let $f\in\LL$ and let $t\in\Rset$ vary; then $\vec{f}(t)=
\Langle\K{n}[f](t)\Rangle_{n\in\Nset}$ is a continuous curve in
$l_2$.
\end{lemma}

\begin{proof} Let $f\in\LL$; then, since
$\sum_{n=0}^{\infty}\K{n}[f](t)^2$ converges to a continuous
(constant) function, by Dini's theorem, it converges uniformly
on every finite interval $I$.  Thus, the last two sums on the
right side of inequality
$\norm{f(t)-f(t+h)}^2\leq\sum_{n=0}^{N}
(\K{n}[f](t)-\K{n}[f](t+h))^2+ 2\sum_{n=N+1}^{\infty}
\K{n}[f](t)^2+2\sum_{n=N+1}^{\infty}\K{n}[f](t+h)^2$ can be
made arbitrarily small on $I$ if $N$ is sufficiently large.
Since functions $\K{n}[f](t)$ have continuous derivatives, they
are uniformly continuous on $I$. Thus,
$\sum_{n=0}^{N}(\K{n}[f](t)-\K{n}[f](t+h))^2$ can also be made
arbitrarily small on $I$ by taking $|h|$ sufficiently small.
\end{proof}

\begin{lemma}[\!\!\cite{IG5}]\label{differentiable} If $g^\prime\in\LL$,
then $\displaystyle{\lim_{|h|\rightarrow 0}
\norm{\frac{g(t)-g(t+h)}{h}-g^\prime(t)} = 0}$; thus, the curve
$\vec{g}(t)=\Langle \K{n}[g](t)\Rangle_{n\in\Nset}$ is
differentiable, and
$(\vec{g})^{\prime}(t)=\Langle\K{n}[g^{\prime}](t)
\Rangle_{n\in\Nset}$.
\end{lemma}

\begin{proof} Let $I$ be any finite interval;
since  $g^{\prime}\in \LL$, for every $\varepsilon >0$ there
exists $N$ such that $\sum_{n=N+1}^{\infty}\K{n}
\left[g^\prime\right](u)^2<\varepsilon /8$ for all  $u\in I$.
Since functions $\K{n} \left[g^\prime\right](u)$ are uniformly
continuous on $I$, there exists a $\delta > 0$ such that for
all $t_1,t_2\in I$, if $|t_1-t_2|<\delta$ then
$\sum_{n=0}^{N}\left(\K{n}
[g^\prime](t_1)-\K{n}[g^\prime](t_2)\right)^2<\varepsilon/2$.
Let $h$ be an arbitrary number such that $|h|<\delta$; then for
every $t$ there exists a sequence of numbers $\xi_n^t$ that lie
between $t$ and $t-h$, and such that
$(\K{n}[g](t)-\K{n}[g](t-h))/h =\K{n}[g^\prime](\xi_n^t)$.
Thus, for all $t\in I$,
\begin{eqnarray*}
&&\sum_{n=0}^{\infty}\K{n}
\left[\frac{g(t)-g(t-h)}{h}-g^\prime(t)\right]^2
=\sum_{n=0}^{\infty}\left(\K{n}
[g^\prime](\xi_n^t)-\K{n}[g^\prime](t)]\right)^2\\
&&\ \ \ \ <\sum_{n=0}^{N}\left(\K{n}
[g^\prime](\xi_n^t)-\K{n}[g^\prime](t)\right)^2+
2\sum_{n=N+1}^{\infty}\K{n} [g^\prime](\xi_n^t)^2\nonumber \\
&&\hspace*{10mm}+
2\sum_{n=N+1}^{\infty}\K{n} [g^\prime](t)^2
<\varepsilon/2+4\varepsilon/8=\varepsilon.\nonumber
\end{eqnarray*}
\end{proof}
\noindent Since $\K{n}[\mm](t)\in\LL$ for all $n$, if we let
$\vec{e}_{k+1}(t)=\Langle(\K{k}\circ\K{n})[\mm](t)
\Rangle_{n\in\Nset}$ for $k\geq 0$, then by
Lemma~\ref{differentiable}, $\vec{e}_k(t)$ are differentiable
for all $k$. Since $l_2$ is complete and $\vec{e}_1(t)$ is
continuous, $\vec{e}_1(t)$ has an antiderivative $\vec{H}(t)$.
Using \eqref{three-term}, we have
\begin{eqnarray*}
\vec{e}_1(t)&=&
\vec{H}^{\!\!\ \prime}(t);\\
\vec{e}_1^{\
\prime}(t)&=&\Langle(\dd\circ\K{0}\circ\K{n})
[\mm](t)\Rangle_{n\in\Nset}=\gamma_0\, \Langle(\K{1}\circ\K{n})
[\mm](t)\Rangle_{n\in\Nset}\\
&=&\gamma_0\, \vec{e}_2(t);\label{FS1}\nonumber \\
\vec{e}_k^{\ \prime}(t)&=&-\gamma_{k-2}
\Langle(\K{k-2}\circ\K{n})
[\mm](t)\Rangle_{n\in\Nset}+\gamma_{k-1}\Langle(\K{k}\circ\K{n})
[\mm](t)\Rangle_{n\in\Nset}\\
&=& - \gamma_{k-2}\,\vec{e}_{k-1}(t) +
\gamma_{k-1}\,\vec{e}_{k+1}(t),\;\;\; \mbox{for $k\geq 2$}.
\label{FSk}\nonumber
\end{eqnarray*}
This means that the curve $\vec{H}(t)$ is a helix in $l_2$
because it has constant curvatures $\kappa_k=\gamma_{k-1}$ for all
$k\geq 1$; the above equations are the corresponding
Frenet--Serret formulas and
$\vec{e}_{k+1}(t)=\Langle(\K{k}\circ\K{n})[\mm](t)\Rangle_{n\in
\Nset}$ for $k\geq 0$ form the orthonormal moving frame of the
helix $\vec{H}(t)$.

\section{Examples}\label{examples}
We now present a few examples of chromatic derivatives and
chromatic expansions, associated with several classical
families of orthogonal polynomials. More details and more
examples can be found in \cite{HB}.

\subsection{Example 1: Legendre polynomials/Spherical Bessel functions}
Let $L_n(\omega)$ be the Legendre polynomials; if we set
$P_{n}^{\scriptscriptstyle{L}}(\omega)=
\sqrt{2n+1}\,L_n(\omega/\pi)$ then
\begin{equation*}
\int_{-\pi}^{\pi}
P_{n}^{\scriptscriptstyle{L}}(\omega)
P_{m}^{\scriptscriptstyle{L}}(\omega)\;\frac{\rm{d} \omega}{2\pi}=\delta(m-n).
\end{equation*}

The corresponding recursion coefficients in equation
\eqref{poly} are given by the formula $\gamma_n=\pi
(n+1)/\sqrt{4(n+1)^2-1}$; the corresponding space $\LT$ is
$L^2[-\pi,\pi]$. The space $\LB$ for this particular example
consists of all entire functions whose restrictions to $\Rset$
belong to $L^2$ and which have a Fourier transform supported in
$[-\pi,\pi]$. Proposition~\ref{local-space-gen} implies that in
this case our locally defined scalar product $\Mscal{f}{g}$,
norm $\norm{f}$ and convolution $(f \ast_{\Mi} g)(t)$ coincide
with the usual scalar product, norm and convolution on $L_2$.

\subsection{Example 2: Chebyshev polynomials of the first
kind/Bessel functions} Let
$P_{n}^{\scriptscriptstyle{T}}(\omega)$ be the family of
orthonormal polynomials obtained by normalizing and rescaling
the Chebyshev polynomials of the first kind, $T_n(\omega)$, by
setting $P_{0}^{\scriptscriptstyle{T}}(\omega)= 1$ and
$P_{n}^{\scriptscriptstyle{T}}(\omega)=\sqrt{2}\;T_n(\omega/\pi)$
for $n>0$.  In this case
\begin{equation*}
\int_{-\pi}^{\pi}P_{n}^{\scriptscriptstyle{T}}(\omega)
P_{m}^{\scriptscriptstyle{T}}(\omega)
\frac{\rm{d}\omega}{
\pi\sqrt{\pi^2-\omega^2}}=\delta(n-m).
\end{equation*}

By Proposition \ref{weight-space}, the corresponding space
$\LB$ contains all entire functions $f(t)$ which have a Fourier
transform $\widehat{f}(\omega)$ supported in $[-\pi,\pi]$ that
also satisfies  $\int_{-\pi}^{\pi}\sqrt{\pi^2-\omega^2}\;
|\widehat{f}(\omega)|^2\rm{d} \omega<\infty$. In this case the
corresponding space $\LB$ contains functions which do not
belong to $L^2$; the corresponding function \eqref{mmz} is
$\mm(z)={\mathrm J}_0(\pi z)$ and  for $n>0$,
$\K{n}[\mm](z)=(-1)^{n}\sqrt{2}\,{\mathrm J}_n(\pi z)$, where
${\mathrm J}_n(z)$ is the Bessel function of the first kind of
order $n$. In the recurrence relation \eqref{three-term} the
coefficients are given by $\gamma_0=\pi/\sqrt{2}$ and
$\gamma_n=\pi/2$ for $n>0$.

The chromatic expansion of a function $f(z)$ is the Neumann
series of $f(z)$ (see \cite{WAT}),
\begin{equation*}
f(t)=f(u){\mathrm J}_0 (\pi(
z-u))+\sqrt{2}\;\sum_{n=1}^{\infty}\K{n}[f](u){\mathrm J}_n(\pi (z-u)).
\end{equation*}
Thus, the chromatic expansions corresponding to various
families of orthogonal polynomials can be seen as
generalizations of the Neumann series, while the families of
corresponding functions $\{\K{n}[\mm](z)\}_{n\in\Nset}$ can be
seen as generalizations and a uniform representation of some
familiar families of special functions.

\subsection{Example 3:  Hermite
polynomials/Gaussian monomial functions}

Let $H_n(\omega)$ be the Hermite polynomials; then polynomials
$P_{n}^{\scriptscriptstyle{H}}(\omega)=
(2^{n}n!)^{-1/2}H_n(\omega)$  satisfy
\begin{equation*}\int_{-\infty}^{\infty}
P_{n}^{\scriptscriptstyle{H}}(\omega)
P_{m}^{\scriptscriptstyle{H}}(\omega)
\;{\e}^{-\omega^2}\;\frac{\rm{d}\omega}{\sqrt{\pi}}
=\delta(n-m).
\end{equation*}
The corresponding space $\LB$ contains entire functions whose
Fourier transform $\widehat{f}(\omega)$ satisfies
$\int_{-\infty}^{\infty} |\widehat{f}(\omega)|^2
\;{\e}^{\omega^2} \rm{d} \omega<\infty$. In this case the space
$\LL$ contains non-bandlimited signals; the corresponding
function defined by \eqref{mmz} is $\mm(z)={\e}^{-z^2/4}$ and
$\K{n}[\mm](z)= (-1)^{n}(2^{n}\,n!)^{-1/2}\,z^{n}
{\e}^{-z^2/4}$. The corresponding recursion coefficients are
given by $\gamma_n=\sqrt{(n+1)/2}$. The chromatic expansion of
$f(z)$ is just the Taylor expansion of $f(z)\e^{z^2/4}$,
multiplied by $\e^{-z^2/4}$.

\subsection{Example 4: Herron family}
This example is a slight modification of an example from
\cite{HB}. Let the family of orthonormal polynomials be
given by the recursion $L_0(\omega)= 1$, $L_1(\omega)=\omega$,
and $L_{n+1}(\omega)=\omega/(n+1)L_n(\omega)-
n/(n+1)L_{n-1}(\omega).$ Then
\begin{equation*}
\frac{1}{2}\int_{-\infty}^{\infty}L(m,\omega)\;L(n,\omega)\;
\sech\left(\frac{\pi\omega}{2}\right)
\;\rm{d}\omega=
\delta(m-n).
\end{equation*}
In this case $\mm(z)=\sech z$ and $\K{n}[\mm](z)= (-1)^{n}\sech
z\, \tanh^{n} z$. The recursion coefficients are given by
$\gamma_n=n+1$ for all $n\geq 0$. If $E_n$ are the Euler
numbers, then $\sech z=\sum_{n=0}^\infty E_{2n}\,z^{2n}/(2n)!$,
with the series converging only in the disc of radius $\pi/2$.
Thus, in this case $\mm(z)$ is not an entire function.

\section{Weakly bounded moment functionals}\label{sec5}
\subsection{}

To study local convergence of chromatic expansions of functions
which are not in $\LL$ we found it necessary to
restrict the class of chromatic moment functionals. The
restricted class, introduced in \cite{IG5}, is still very broad
and contains functionals
that correspond to many classical families of orthogonal
polynomials. It consists of functionals such that the
corresponding recursion coefficients $\gamma_n>0$ appearing in
\eqref{three-term} are such that
sequences $\{\gamma_n\}_{n\in \Nset}$ and
$\{\gamma_{n+1}/\gamma_{n}\}_{n\in\Nset}$ are bounded from
below by a positive constant, and such that the growth rate of the sequence
$\{\gamma_n\}_{n\in \Nset}$ is sub-linear in $n$.  For
technical simplicity in the definition below these conditions
are formulated using a single constant $M$ in all of the bounds.
\begin{definition}[\!\!\cite{IG5}]\label{def-weak} Let \mom\ be a
moment functional such that for some $\gamma_n>0$
\eqref{three-term} holds.
\begin{enumerate} \item \mom\ is \textit{weakly bounded}
if there exist some $M\geq 1$, some $0\leq p<1$ and some
integer $r\geq 0$, such that for all $n\geq 0$,
\begin{equation}
\frac{1}{M}\leq \gamma_n\leq M (n + r)^p, \label{one-weak}\\
\end{equation}
\begin{equation}\frac{{\gamma_{n}}}{{\gamma_{n+1}}}\leq M^2.\label{two-weak}
\end{equation}
\item \mom\ is \textit{bounded} if there exists some $M\geq
    1$ such that for all $n\geq 0$,
\begin{equation}
\frac{1}{M}\leq{\gamma_n}\leq M.
\end{equation}
\end{enumerate}
\end{definition}

Since \eqref{three-term} is assumed to hold for some
$\gamma_n>0$, weakly bounded moment functionals are positive
definite and symmetric. Every bounded functional \mom\ is also
weakly bounded with $p=0$. Functionals in our Example 1 and
Example 2 are bounded. For bounded moment functionals \mom\ the
corresponding moment distribution $ a(\omega)$ has a finite
support \cite{Chih} and consequently $\mm(t)$ is a band-limited
signal. However, $\mm(t)$ can be of infinite energy (i.e., not
in $L^2$) as is the case in our Example 2. Moment functional in
Example 3 is weakly bounded but not bounded $(p=1/2)$; the
moment functional in Example 4 is not weakly bounded $(p=1)$.
We note that important examples of classical orthogonal
polynomials which correspond to weakly bounded moment
functionals in fact satisfy a stronger condition from the
following simple Lemma.

\begin{lemma}
Let \mom\ be such that \eqref{three-term} holds for some $\gamma_n>0$.
If for some $0\leq p<1$ the sequence ${\gamma_n}/{n^p}$
converges to a finite positive limit, then \mom\ is weakly
bounded.
\end{lemma}

Weakly bounded moment functionals allow a useful estimation of
the coefficients  in the corresponding equations
\eqref{inverse} and \eqref{direct} relating the chromatic and
the ``standard'' derivatives.

\begin{lemma}[\!\!\cite{IG5}]\label{bounds}
Assume that $\mom$ is such that for some $M\geq 1, r\geq 0$ and
$p\geq 0$ the corresponding recursion coefficients $\gamma_n$
for all $n$ satisfy inequalities \eqref{one-weak}. Then the
following inequalities hold for all $k$ and $n$:
\begin{eqnarray}
|(\K{n}\circ
\dd^k)[\mm](0)|&\leq&(2M)^k(k+r)!^p;\label{b-bound}\\
\left|\K{n} \left[\frac{t^k}{k!}\right](0)\right|
&\leq&(2M)^{n}.\label{mono-bound}
\end{eqnarray}
\end{lemma}

\begin{proof}
By \eqref{bkn}, it is enough to prove \eqref{b-bound} for all
$n,k$ such that $n\leq k$.  We proceed by induction on $k$,
assuming the statement holds for all $n\leq k$. Applying
\eqref{three-term} to $\dd^{k}[\mm](t)$ we get
\begin{equation*}|(\K{n}\circ
\dd^{k+1})[\mm](t)|\leq\! {\gamma_{n}}|(\K{n+1}\circ
\dd^{k})[\mm](t)|+
{\gamma_{n-1}}\,|(\K{n-1}\circ \dd^{k})[\mm](t)|.
\end{equation*}
Using the induction hypothesis and \eqref{bkn} again, we get
for all $n\leq k+1$,
\begin{eqnarray*}|(\K{n}\circ \dd^{k+1})
\left[\mm\right](0)|
&< & (M(k+1+r)^p+M(k+r)^p) (2M)^{k} (k+r)!^p\nonumber\\
&<&(2M)^{k+1} (k+1+r)!^p.
\end{eqnarray*}
Similarly, by \eqref{zero-mon}, it is enough to prove
\eqref{mono-bound} for all $k\leq n$. This time we proceed by
induction on $n$ and use \eqref{three-term}, \eqref{one-weak}
and \eqref{two-weak} to get
\begin{equation*}
\left|\K{n+1}\left[\frac{t^k}{k!}\right]\right| \leq M
\left|\K{n}\left[\frac{t^{k-1}}{(k-1)!}\right]\right|+ M^2\,
\left|\K{n-1}\left[\frac{t^k}{k!}\right]\right|.
\end{equation*}
By induction hypothesis and using \eqref{zero-mon} again, we
get that for all $k\leq n+1$,
$\left|\K{n+1}\left[\frac{t^k}{k!}\right](0)\right|< M\,
(2M)^{n}+ M^2 (2M)^{n-1}<(2M)^{n+1}$.
\end{proof}

\begin{corollary}[\!\!\cite{IG5}]\label{poly-bdd} Let $\mom$ be weakly bounded;
then for every fixed $n$
\[
\LIM{k}\left|(\K{n}\circ \dd^{k})
\left[\mm\right](0)/k!\right|^{1/k}=0,
\]
and the convergence is uniform in $n$.
\end{corollary}

\begin{proof}  Let $R(k)=(k+r)!/k!$; then $R(k)$ is a polynomial
of degree $r$, and, by \eqref{b-bound},
\begin{equation}\label{poly-bound-eq}
\left|\frac{(\K{n}\circ \dd^k)\left[\mm\right](0)}{k!}\right|^{1/k}\leq
\frac{2M R(k)^{p/k}}{k!^{(1-p)/k}} <\frac{2M{\e}^{1-p}\; R(k)^{p/k}}{k^{1-p}}.
\end{equation}
\end{proof}

\begin{corollary}[\!\!\cite{IG5}]\label{moments-asymptotics}
Let $\mm(z)$ correspond to a weakly bounded moment functional
\mom; then
\begin{equation}\label{entire}
\LIM{k}\left(\frac{\mu_k}{k!}\right)^{1/k}=
\LIM{k}\left|\frac{\mm^{k}(0)}{k!}\right|^{1/k}
=0.\end{equation}
Thus, since \eqref{limsup} is satisfied with $\rho=0$,
every weakly bounded moment functional is chromatic.
\end{corollary}
Note that this and Proposition~\ref{anaz} imply that $\mm(z)$
is an entire function. If \eqref{one-weak} holds with $p=1$,
then Lemma~\ref{bounds} implies only
\begin{equation}\label{p=1}
\limsup_{k\rightarrow\infty}\left|\frac{(\K{n}\circ
\dd^{k})\left[\mm\right](0)}{k!}\right|^{1/k}\leq 2M.
\end{equation}
Example 4 shows that  in this case the corresponding function
$\mm(z)$ need not be entire. Thus, if we are interested in
chromatic expansions of entire functions, the upper bound in
\eqref{one-weak} of the definition of a weakly bounded moment
functional is sharp.

Lemma~\ref{bounds}  and Proposition~\ref{complete}
imply the following Corollary.
\begin{corollary}[\!\!\cite{IG5}]
If \mom\ is weakly bounded, then the corresponding family of
polynomials $\PPP$ is a complete system in $\LT$.
\end{corollary}
Thus, we get that the Chebyshev, Legendre,  Hermite and
similar classical families of orthogonal polynomials are complete
in their corresponding spaces $\LT$.

To simplify our estimates, we choose $K\geq 1$ such that for
$p$, $M$ and $r$ as in Definition~\ref{def-weak} for all $k>0$,
we have
\begin{equation}\label{K}
\frac{(2M)^{k} (k + r)!^{p}}{k!^p} <K^{k}.
\end{equation}
The following Lemma slightly improves a result from \cite{IG5}.
\begin{lemma}\label{bounds1} Let \mom\  be weakly bounded and
$p<1$ and $K\geq 1$ such that \eqref{one-weak} and \eqref{K}
hold. Let also $k$ be an integer such that $k\geq
1/(1-p)$. Then there exists a polynomial $P(x)$ of degree $k-1$
such that for every $n$ and every $z\in \Cset$,
\begin{equation}\label{bound-Knb}
|\K{n}[\mm](z)|< \frac{|K z|^{n}}{n!^{1-p}}P(|z|)\,{\e}^{|K z|^k}.
\end{equation}
\end{lemma}

\begin{proof} Using the Taylor series for $\K{n}[\mm](z)$,
\eqref{bkn},  \eqref{b-bound} and \eqref{K}, we get that for $z$ such
that $|Kz|\geq1$,
\begin{eqnarray*}
|\K{n}[\mm](z)|&<&\sum_{m=0}^{\infty}\frac{|K
z|^{n+m}}{(n+m)!^{1-p}}\ \leq\ \frac{|K z|^{n}}{n!^{1-p}}
\sum_{m=0}^{\infty}\frac{|K z|^{m}}{m!^{1/k}} \\
&<& \frac{|K z|^{n}}{n!^{1-p}}
\sum_{m=0}^{\infty}\frac{|K z|^{k\lfloor m/k\rfloor+k-1}}{\lfloor
m/k\rfloor !}\\
&=& k\;\frac{|K z|^{n+k-1}}{n!^{1-p}}
\sum_{j=0}^{\infty}\frac{|K z|^{kj}}{j!}\\
& = & k\;\frac{|K z|^{n+k-1}}{n!^{1-p}}\,
{\e}^{|Kz|^k}.
\end{eqnarray*}
If $|K z|<1$, then a similar calculation shows that for such $z$
we have $|\K{n}[\mm](z)|<k\,\e{|K z|^{n}}/{n!^{1-p}}$.
The claim now follows with
$P(|z|)=k(|K z|^{k-1}+\e)$.
\end{proof}

\subsection{Local convergence of chromatic expansions}
\begin{proposition}\label{expansionK}
Let \mom\ be weakly bounded, $p<1$ as in \eqref{one-weak},
$f(z)$ a function analytic on a domain $G\subseteq\Cset$ and
$u\in G$.
\begin{enumerate}
\item If the sequence $|\K{n}(u)/n!^{1-p}|^{1/n}$ is bounded,
    then the chromatic expansion $\CE[f,u](z)$ of
    $f(z)$ converges uniformly to $f(z)$ on
    a disc $D\subseteq G$, centered at $u$.

\item In particular, if $|\K{n}(u)/n!^{1-p}|^{1/n}$
    converges to zero, then the chromatic expansion
    $\CE[f,u](z)$ of $f(z)$ converges for every $z\in G$
    and the convergence is uniform on every finite closed
    disc around $u$, contained in $G$.
\end{enumerate}
\end{proposition}

\begin{proof}
Assume that $R$ is such that $\limsup_{n\rightarrow
\infty}|\K{n}[f](u)/n!^{1-p}|^{1/n}<R$. Then
$|\K{n}[f](u)|<R^nn!^{1-p}$ for all sufficiently large $n$. Let
$K$ and $k$ be such that \eqref{bound-Knb} holds;
then Lemma \ref{bounds1} implies that for all sufficiently large $n$,
\[
|\K{n}[f](u)\K{n}[\mm](z-u)|^{1/n}<RK(P(|z-u|){\e}^{|K(z-u)|^k})^{1/n}|z-u|.
\]
Thus, the chromatic series converges uniformly inside every disc
$D\subseteq G$,  centered at $u$, of radius less than $1/(RK)$.
Since
\begin{equation*}
\K{j}[\CA[f,u](z)]\big|_{z=u}=\sum_{n=0}^{\infty}
(-1)^n\K{n}[f](u)(\K{j}\circ\K{n})[\mm](0)=\K{j}[f](u),
\end{equation*}
$\CA[f,u](z)$ converges to $f(z)$ on $D$.
\end{proof}
\begin{lemma} Let $M$ be as in Definition \ref{def-weak}. Then
\[
\limsup_{n\rightarrow \infty}\left|\frac{\K{n}[f](u)}{n!^{1-p}}
\right|^{1/n} \leq 2M\limsup_{n\rightarrow \infty}
\left|\frac{f^{(n)}(u)}{n!^{1-p}}\right|^{1/n}.
\]
\end{lemma}
\begin{proof}
Let $\beta>0$ be any number such that
$\limsup_{n\rightarrow
\infty}|{f^{(n)}(u)}/{n!^{1-p}}|^{1/n}<\beta$; then there
exists $B_{\beta}\geq 1$ such that $|f^{(k)}(u)|\leq
B_{\beta}\;k!^{1-p}\beta^k$ for all $k$.  Using \eqref{direct}
and \eqref{mono-bound} we get
\begin{eqnarray*}
|\K{n}[f](u)| &\leq& \sum_{k=0}^n\left|\K{n}
\left[\frac{t^k}{k!}\right](0)\right||f^{(k)}(u)|
\leq (2M)^nB_{\beta}\sum_{k=0}^n\;k!^{1-p}\beta^k\\
&<& (2M)^nB_{\beta}\sum_{k=0}^n\;\left(\frac{n}{e}\right)^{(k+1)(1-p)}\beta^k.
\end{eqnarray*}
Summation of the last series shows that
$|\K{n}[f](u)|<2B_{\beta}(2M\beta)^nn!^{1-p}$ for sufficiently large $n$.
\end{proof}

\begin{corollary}\label{expansionD}
Let \mom\ be weakly bounded, $p<1$ as in \eqref{one-weak},
$f(z)$ a function analytic on a domain $G\subseteq\Cset$ and
$u\in G$. \begin{enumerate}
\item
    If the sequence $|f^{(n)}(u)/n!^{1-p}|^{1/n}$ is bounded,
    then the chromatic expansion $\CE[f,u](z)$ of
    $f(z)$ converges uniformly to $f(z)$ on
    a disc $D\subseteq G$ centered at $u$ .

\item
    In particular, if $|f^{(n)}(u)/n!^{1-p}|^{1/n}$ converges
    to zero, then the chromatic expansion $\CE[f,u](z)$ of $f(z)$
    converges for all $z\in G$ and the convergence
    is uniform on every closed disc around $u$, contained in $G$.
\end{enumerate}
\end{corollary}

\begin{corollary}[\!\!\cite{IG5}]\label{bounded-exp} If \mom\ is
bounded, then for every entire function $f$ and all $u,z\in
\Cset$, the chromatic expansion $CE[f,u](z)$ converges to
$f(z)$ for all $z$, and the convergence is uniform on every
disc around $u$ of finite radius.
\end{corollary}

\begin{proof}
If $f(z)$ is entire, then for every $u$, $\LIM{n}
|f^{(n)}(u)/n!|^{1/n}=0$. The Corollary now follows from
Corollary~\ref{expansionD} with $p=0$.
\end{proof}

\begin{proposition}\label{exp-type}
Assume \mom\ is weakly bounded and let $0\leq p <1$ be such
that \eqref{one-weak} holds, and $k$ such that $k\geq 1/(1-p)$.
Then there exists $C,L>0$ such that $|f(z)|\leq
C\norm{f}{\e}^{L|z|^k}$ for all $f(z)\in\LB$.
\end{proposition}

\begin{proof}
Since $f(z)\in\LB$, the chromatic expansion of $f(z)$ and
\eqref{bound-Knb} yield
\begin{eqnarray*}
|f(z)|&\leq&\left(\sum_{n=0}^{\infty}|\K{n}[f](0)|^2
\sum_{n=0}^{\infty}|\K{n}[\mm](z)|^2\right)^{1/2}\\
&\leq& \norm{f}P(|z|){\e}^{|Kz|^k}
\left(\sum_{n=0}^{\infty}\frac{|K z|^{2n}}{n!^{2(1-p)}}\right)^{1/2},
\end{eqnarray*}
which, using the method from the proof of Lemma~\ref{bounds1},
can easily be shown to imply our claim.
\end{proof}

Note that for bounded moment functionals, such as those
corresponding to the Legendre or the Chebyshev polynomials, we
have $p=0$; thus, Proposition \ref{exp-type} implies that
functions which are in \LB\ are of exponential type. For \mom\
corresponding to the Hermite polynomials p=1/2 (see Example 3);
thus, we get that there exists $C,L>0$ such that $|f(z)|\leq
C\norm{f}{\e}^{L|z|^2}$ for all $f\in\LB$. It would be
interesting to establish when the reverse implication is true
and thus obtain a generalization of the Paley-Wiener Theorem
for functions satisfying $|f(z)|<C{\e}^{L|z|^k}$ for $k>1$.

\subsection{Generalizations of some classical equalities for
the Bessel functions}

Corollaries \ref{bounded-exp} and \ref{expansionD} generalize
the classic result that every entire function can be expressed
as a Neumann series of Bessel functions \cite{WAT}, by
replacing the Neumann series with a chromatic expansion that
corresponds to any (weakly) bounded moment functional. Thus,
many classical results on Bessel functions from \cite{WAT}
immediately follow from Corollary \ref{bounded-exp}, and, using
Corollary \ref{expansionD}, generalize to functions
$\K{n}[\mm](z)$ corresponding to any weakly bounded moment
functional \mom. Below we give a few illustrative examples.

\begin{corollary}\label{eiw}
Let $\PP{n}{\omega}$ be the orthonormal polynomials associated
with a weakly bounded moment functional \mom; then for every
$z\in\Cset$,
\begin{equation}
{\e}^{{\ii}\omega z}=\sum_{n=0}^{\infty}{\ii}^{n}
\PP{n}{\omega}\,\K{n}[\mm](z).
\end{equation}
\end{corollary}
\begin{proof}
If $p<1$ then
\begin{equation*}
\LIM{n}\frac{\sqrt[n]{|\frac{{\mathrm d}^{n}}{{\mathrm d}
z^{n}}{\e}^{{\ii}\omega z}|_{z=0}}}{n^{1-p}}=\LIM{n}
\frac{|\omega|}{n^{1-p}}=0,
\end{equation*}
and the claim follows from Proposition \ref{expansionD} and
\eqref{iwt}.
\end{proof}
Corollary \ref{eiw} generalizes the well known equality for the
Chebyshev polynomials $T_{n}(\omega)$ and the Bessel functions
${\mathrm J}_{n}(z)$, i.e.,
\[
{\e}^{\ii \omega z}={\mathrm J}_0(z) + 2\sum_{n=1}^{\infty}
{\ii}^{n}T_{n}(\omega){\mathrm J}_{n}(z).
\]

In Example 3, \ref{eiw} becomes the equality for the Hermite
polynomials $H_n(\omega)$:
\[
{\e}^{\ii \omega z}=\sum_{n=0}^{\infty}
\frac{H_{n}(\omega)}{n!}\left(\frac{\ii z}{2}\right)^{n}
{\e}^{-\frac{z^2}{4}}.
\]

By applying Corollary \ref{expansionD} to the constant function
$f(t)\equiv 1$, we get that its chromatic expansion yields that
for all $z\in\Cset$
\begin{equation*}
\mm(z)+\sum_{n=1}^{\infty}\left(\prod_{k=1}^{n}
\frac{\gamma_{2k-2}}{\gamma_{2k-1}}\right)\K{2n}[\mm](z)=1,
\end{equation*}
with $\gamma_n$ the recursion coefficients from \eqref{poly}.
This equality generalizes the equality
\[
{\mathrm J}_{0}(z)+2\sum_{n=1}^{\infty}
{\mathrm J}_{2n}(z)=1.
\]

Using Proposition \ref{unif-con} to expand $\mm(z+u)\in\LB$
into chromatic series around $z=0$, we get that for all
$z,u\in\Cset$
\begin{equation*}\mm(z+u)=\sum_{n=0}^{\infty}(-1)^n
\K{n}[\mm](u)\K{n}[\mm](z),
\end{equation*}
which generalizes the equality
\[{\mathrm J}_0(z+u)={\mathrm J}_0(u){\mathrm J}_0(z)+2
\sum_{n=1}^{\infty}(-1)^n{\mathrm J}_n(u){\mathrm J}_n(z).
\]

\section{Some Non-Separable Spaces}\label{nonsep}

\subsection{}
Let \mom\ be weakly bounded; then periodic functions do not
belong to $\LL$ because $\sum_{n=0}^{\infty} \K{n}[f](t)^2$
diverges. We now introduce some nonseparable inner product
spaces in which pure harmonic oscillations have finite norm
and are pairwise orthogonal.\\

\noindent\textbf{Note:} \textit{In the remainder of
this paper we consider only weakly bounded moment functionals
and real functions which are restrictions of entire functions.}\\

\begin{definition}
Let let $0\leq p<1$
be as in \eqref{one-weak}. We denote by $\CC$ the vector space
of functions such that the sequence
\begin{equation}\label{nu}
\nu_{n}^{f}(t)=\frac{1}{(n+1)^{1-p}}\sum_{k=0}^{n} \K{k}[f](t)^2
\end{equation}
converges uniformly on every finite interval $I\subset\Rset$.
\end{definition}
\begin{proposition} Let $f,g\in\CC$ and
\begin{equation}\label{sigma}
\sigma_{n}^{fg}(t)=\frac{1}{(n+1)^{1-p}}\sum_{k=0}^{n}
\K{k}[f](t)\K{k}[g](t);
\end{equation}
then the sequence $\{\sigma_{n}^{fg}(t)\}_{n\in\Nset}$
converges to a constant function. In particular,
$\{\nu_{n}^{f}(t)\}_{n\in\Nset}$ also converges to a constant
function.
\end{proposition}

\begin{proof}
Since $\nu_{n}^{f}(t)$ and $\nu_{n}^{g}(t)$  given by \eqref{nu}
converge uniformly on every finite interval, the same holds for the
sequence $\sigma_{n}^{fg}(t)$. Consequently, it is enough to
show that for all $t$, the derivative $\sigma_{n}^{fg}(t)^{\prime}$
of $\sigma_{n}^{fg}(t)$  satisfies $\LIM{n}{
\sigma_{n}^{fg}(t)^{\prime}} = 0$. Let
\[
S_k(t)=\K{k}[f](t)^2+ \K{k+1}[f](t)^2+
\K{k}[g](t)^2+\K{k+1}[g](t)^2;
\]
then, since $f,g\in\CC$, the sequence
${1}/{(n+1)^{1-p}}\sum_{k=0}^{n}S_k(t)$
converges everywhere to some $\alpha(t)$. We now show that if
$t$ is such that $\alpha(t)>0$, then there are infinitely many
$k$ such that $S_k(t)<2\alpha(t)k^{-p}$. Assume opposite, and
let $K$ be such that $S_k(t)\geq 2\alpha(t)k^{-p}$ for all
$k\geq K$. Then, since
\[
\sum_{k=K}^{n}k^{-p}>\int_{K}^{n+1}x^{-p}{\mathrm d}
x=\frac{(n+1)^{1-p}-K^{1-p}}{1-p},
\]
we would have that for all $n>K$,
\begin{eqnarray*}
\frac{\sum_{k=K}^{n}S_k(t)}{(n+1)^{1-p}}\geq
\frac{2\alpha(t)\sum_{k=K}^{n}k^{-p}}{(n+1)^{1-p}}
> \frac{2\alpha(t)((n+1)^{1-p}-K^{1-p})}{(n+1)^{1-p}(1-p)}.
\end{eqnarray*}
However, since $0\leq p<1$, this would imply
${\sum_{k=0}^{n}S_k(t)}/{(n+1)^{1-p}}>\alpha(t)$ for all
sufficiently large $n$, which contradicts the definition of
$\alpha(t)$. Consequently, for infinitely many $n$ all four
summands in $S_n(t)$ must be smaller than $2\alpha(t)\,n^{-p}$.
For those values of $n$ we have
\[
|\K{n+1}[f](t)\,\K{n}[g](t)|+|\K{n}[f](t)\,\K{n+1}[g](t)|
<4\alpha(t)n^{-p}.
\]

Since \mom\ is weakly bounded, \eqref{C-D} and \eqref{one-weak}
imply that for some $M\geq 1$ and an integer $r$,
\begin{equation*}
\left|\sigma_{n}^{fg}(t)^{\prime}\right|
< \frac{M (n+r)^p}{(n+1)^{1-p}}
(|\K{n+1}[f](t)\,\K{n}[g](t)|+|\K{n}[f](t)\,\K{n+1}[g](t)|).
\end{equation*}
Thus, for infinitely many $n$ we have
\begin{equation*}
\left|\sigma_{n}^{fg}(t)^{\prime}\right|< \frac{4M\,
 (n+r)^p\; n^{-p}\,\alpha(t)}{(n+1)^{1-p}}.
\end{equation*}
Consequently, $\liminf_{n\rightarrow\infty}\left|
\sigma_{n}^{fg}(t)^{\prime}\right|=0$ and since $\LIM{n}
\sigma_{n}^{fg}(t)^{\prime}$ exists, it must be equal to zero.
\end{proof}

\begin{corollary}
Let $\CC_0$ be the vector space consisting of
functions $f(t)$ such that $\LIM{n}\nu_{n}^{f}(t)=0$; then in
the quotient space $\CC_2=\CC/\CC_0$ we can introduce a scalar
product by the following formula whose right hand side is
independent of $t$:
\begin{eqnarray}\label{scal-cesaro}
\Scal{f}{g}=\LIM{n}\frac{1}{(n+1)^{1-p}}
\sum_{k=0}^{n} \K{k}[f](t)\, \K{k}[g](t).
\end{eqnarray}
\end{corollary}

The corresponding norm on $\CC_2$ is denoted by
$\Norm{\,\cdot\,}$. Clearly, all real valued functions from
$\LL$ belong to $\CC_0$.

\begin{proposition}
If $f\in\CC_2$, then the chromatic expansion of $f(t)$
converges to $f(t)$ for every $t$ and the convergence is
uniform on every finite interval.
\end{proposition}
\begin{proof}
Since ${1}/{(n+1)^{1-p}}\sum_{k=0}^{n} \K{k}[f](t)^2$ converges
to $0<(\Norm{f})^2<\infty$, for all sufficiently large $n$,
\begin{eqnarray*}
|\K{n}[f](t)|^{1/n}&\leq& \left(\sum_{k=0}^{n}
\K{k}[f](t)^2\right)^{1/(2n)}\\ &\leq&
(2 \Norm{f})^{1/n}(n+1)^{(1-p)/(2n)}.
\end{eqnarray*}
Thus, $|\K{n}[f](t)/n!^{1-p}|^{1/n}\rightarrow 0$, and the
claim follows from Proposition \ref{expansionK}.
\end{proof}

Since $\K{n}[\mm](t)\in\LL$, we have $\Norm{\sum_{j=0}^{n}
(-1)^j \K{j}[f](0)\,\K{j}[\mm](t)}=0$ for all $n$. Thus, the
chromatic expansion of $f\in\CC_2$ does not converge to $f(t)$
in $\CC_2$. Moreover, there can be no such series
representation of functions $f\in\CC_2$, converging in $\CC_2$,
because the space $\CC_2$ is in general nonseparable, as the
remaining part of this paper shows.

\subsection{Space $\CC_2$ associated with the Chebyshev polynomials (Example 2)}\label{cheb}

For this case the corresponding space $\CC_2$ will be denoted
by $\CT$, and in \eqref{one-weak} we have $p=0$. Thus, the
scalar product on $\CT$ is defined by
\[\ScalT{f}{g}=\LIM{n}\frac{1}{n+1}
\sum_{k=0}^{n}\K{k}[f](t)\K{k}[g](t).\]

\begin{proposition}\label{CE}
Functions $f_{\omega}(t)=\sqrt{2}\,\sin\omega t$ and
$g_{\omega}(t)=\sqrt{2}\,\cos\omega t$ for $0<\omega <\pi$ form
an orthonormal system of vectors in $\CT$.
\end{proposition}
\begin{proof}

From \eqref{iwt} we get
\begin{eqnarray}\label{sin}
\Scal{f_{\omega} }{f_{\sigma}}\hspace*{-3mm}&=&
\LIM{n}\frac{\sum_{k=0}^{n}
\PT{2k}{\omega}\PT{2k}{\sigma}\sin\omega t\sin\sigma t}{2n+1}
\nonumber\\
&&\hspace*{10mm} +\LIM{n}\frac{\sum_{k=0}^{n-1}
\PT{2k+1}{\omega}\PT{2k+1}{\sigma}
\cos\omega t\cos\sigma t}{2n+1}.
\end{eqnarray}
Since $\PT{n}{\omega}\leq\sqrt{2}$ on $(0,\pi)$, \eqref{CDP}
implies that for $\omega,\sigma\in (0,\pi)$ and
$\omega\neq\sigma$,
\begin{equation*}
\LIM{n}\frac{\sum_{k=0}^{n}
\PT{k}{\omega}\PT{k}{\sigma}}{n+1}
=\LIM{n}\frac{\PT{{n+1}}{\omega}
\PT{{n}}{\sigma}-\PT{{n+1}}{\sigma}\PT{{n}}{\omega}}{2(n+1)}=0.
\end{equation*}
Since $\PT{2n}{\omega}$ are even functions and
$\PT{2n+1}{\omega}$ odd, this also implies that
\begin{equation*}\LIM{n}\frac{\sum_{k=0}^{n}
\PT{2k}{\omega}\PT{2k}{\sigma}}{2n+1}=
\LIM{n}\frac{\sum_{k=0}^{n-1}
\PT{2k+1}{\omega}\PT{2k+1}{\sigma}}{2n+1}=0.
\end{equation*}
Thus, by \eqref{sin}, $\Scal{f_{\omega} }{f_{\sigma}}=0$. Using
\eqref{sumsquares},  one can verify that for $0<\omega<\pi$
\begin{equation*}
\frac{1}{n+1}\sum_{k=0}^{n}\PT{{k}}{\omega}^2 =
\frac{1+2n}{2n+2}+\frac{\sin((2n+1) \arccos\omega)}
{(2n+2)\sqrt{1-\omega^2}}\rightarrow 1.
\end{equation*}
Thus, \eqref{equal} implies that for $0<\omega<\pi$
\begin{equation*}
\LIM{n}\frac{1}{2n+1}\sum_{k=0}^{n}\PT{{2k}}{\omega}^2 =
\LIM{n}\frac{1}{2n+1}\sum_{k=0}^{n}\PT{{2k+1}}{\omega}^2
=\frac{1}{2}
\end{equation*}
Consequently, $\Norm{\sqrt{2}\sin\omega t}=1$.
\end{proof}

\subsection{Space $\CC_2$ associated with the Hermite polynomials (Example 3)}\label{herm}

The corresponding space $\CC_2 $ in this case is denoted by
$\CH$, and in \eqref{one-weak} we have $p=1/2$. Thus, the
scalar product in $\CH$ is defined by
\[\ScalH{f}{g}=\LIM{n}\frac{1}{\sqrt{n+1}}\;
\sum_{k=0}^{n}\K{k}[f](t)\K{k}[g](t).\]

\begin{proposition}\label{HE}
For all $\omega >0 $ functions $f_{\omega}(t)=\sin\omega t$ and
$g_{\omega}(t)=\cos\omega t$ form an orthogonal system in
$\CH$, and
$\Norm{f_\omega}=\Norm{g_\omega}={\e^{\omega^2/2}}/
{\sqrt[4]{2\pi}}$.
\end{proposition}

\begin{proof}
For all $\omega$ and for $n\rightarrow \infty$,
\[
\PH{n}{\omega}-
\frac{{\Gamma(n+1)}^{\frac{1}{2}}}{2^{\frac{n}{2}}
\Gamma(\frac{n}{2}+1)}\,
{\e}^{\frac{\omega^2}{2}}\cos\left(\sqrt{2n+1}\;\omega-
\frac{n\pi}{2}\right)\rightarrow 0;
\]
see, for example, 8.22.8 in \cite{SZG}. Using the
Stirling formula we get
\begin{equation}\label{asym}
\PH{n}{\omega}-\left(\frac{2}{\pi}\right)^
{\frac{1}{4}}\;n^{-\frac{1}{4}}\;
{\e}^{\frac{\omega^2}{2}}\cos\left(\sqrt{2n+1}\;\omega-
\frac{n\pi}{2}\right)\rightarrow 0.
\end{equation}
This fact and \eqref{CDP} are easily seen to imply that
$\Scal{f_{\omega} }{f_{\sigma}}=0$ for all distinct
$\omega,\sigma>0$, while \eqref{asym} and \eqref{equal} imply
that
\begin{equation}\label{eq}
\LIM{n}\left(\frac{\sum_{k=0}^{n}
\PH{{2k}}{\omega}^2}{\sqrt{2n+1}}-
\frac{\sum_{k=0}^{n-1}\PH{{2k+1}}{\omega}^2}{\sqrt{2n+1}}\right)
=0.
\end{equation}
Since $H_n^\prime(\omega)=2\;n\;H_{n-1}(\omega)$, we have
$\PH{n}{\omega}^\prime=\sqrt{2\;n}\;\PH{n-1}{\omega}$. Using
this, \eqref{asym} and \eqref{sumsquares} one can verify that
\begin{equation*}
\LIM{n}\frac{\sum_{k=0}^{n}\PH{{k}}{\omega}^2}{\sqrt{n+1}}=
\sqrt{\frac{2}{\pi}}\;{\e}^{\omega^2},
\end{equation*}
which, together with \eqref{eq}, implies that
$\Norm{f_\omega}={\e^{\omega^2/2}}/ {\sqrt[4]{2\pi}}.$
\end{proof}

Note that in this case, unlike the case of the family
associated with the Chebyshev polynomials, the norm of a pure
harmonic oscillation of unit amplitude depends on its
frequency.

One can verify that propositions similar to
Proposition~\ref{CE} and Proposition~\ref{HE} hold for other
classical families of orthogonal polynomials, such as the
Legendre polynomials. Our numerical tests indicate that the
following conjecture is true.\footnote{We have tested this
Conjecture numerically, by setting $\gamma_n=n^p$ for several
values of $p<1$, and in all cases a finite limit appeared to
exist. Paul Nevai has informed us that the special case of this
Conjecture for $p=0$ is known as Nevai-Totik Conjecture, and is
still open.}

\begin{conjecture}\label{conj1} Assume that for some $p<1$ the
recursion coefficients $\gamma_n$ in \eqref{poly} are such that
$\gamma_n/n^p$ converges to a finite positive limit. Then, for
the corresponding family of orthogonal polynomials we have
\begin{equation}\label{hyp}
0< \LIM{n}\frac{1}{(n+1)^{1-p}}\sum_{k=0}^{n}\PP{k}{\omega}^2 <\infty
\end{equation}
for all $\omega$ in the support $sp(a)$  of the corresponding
moment distribution function $a(\omega)$. Thus, in the corresponding
space $\CC_2$ all pure harmonic oscillations with positive
frequencies $\omega$ belonging to the support of the
moment distribution $a(\omega)$ have finite positive norm and are
mutually orthogonal.
\end{conjecture}

Note that \eqref{sumsquares} implies that \eqref{hyp} is
equivalent to
\[
0< \LIM{n}\frac{\PP{{n+1}}{\omega}^\prime
\PP{{n}}{\omega}-\PP{{n+1}}{\omega}
\PP{{n}}{\omega}^\prime}{(n+1)^{1-2p}}<\infty.
\]

\section{Remarks}
The special case of the chromatic derivatives presented in
Example 2 were first introduced in \cite{IG0}; the
corresponding chromatic expansions were subsequently introduced
in \cite{IG00}. These concepts emerged in the course of the
author's design of a pulse width modulation power amplifier.
The research team of the author's startup, \emph{Kromos
Technology Inc.,} extended these notions to various systems
corresponding to several classical families of orthogonal
polynomials \cite{CH,HB}. We also designed and implemented a
channel equalizer \cite{H} and a digital transceiver
(unpublished), based on chromatic expansions. A novel image
compression method motivated by chromatic expansions was
developed in \cite{C0,C1}. In \cite{CNV} chromatic expansions
were related to the work of Papoulis \cite{Pap} and
Vaidyanathan \cite{VA}. In \cite{NIV} and \cite{VIN} the theory
was cast in the framework commonly used in signal processing.
Chromatic expansions were also studied in \cite{CH}, \cite{B}
and and \cite{WS}. Local convergence of chromatic expansions
was studied in \cite{IG5}; local approximations based on
trigonometric functions were introduced in \cite{IG6}.  A
generalization of chromatic derivatives, with the prolate
spheroidal wave functions replacing orthogonal polynomials, was
introduced in \cite{Gil}; the theory was also extended to the
prolate spheroidal wavelet series that combine
chromatic series with sampling series.\\

\noindent\textbf{Note:} Some Kromos technical reports and some
manuscripts can be found at the author's web site
\texttt{http://www.cse.unsw.edu.au/\~{}{ignjat}/diff/}.

\bibliographystyle{plain}

\end{document}